\documentclass{amsart}

\usepackage{amssymb}

\usepackage{xy}
\xyoption{arrow}
\xyoption{matrix}
\xyoption{cmtip}
\SelectTips{cm}{}

\usepackage{graphicx}
\usepackage{placeins}

\newcommand{\whd}{Wh^{\Diff}}
\newcommand{\fu}{Cu}
\newcommand{\fubar}{\overline{Cu}\strut}

\newcommand{\tK}{K^{\mathrm{red}}}

\newcommand{\trc}[1]{\mathit{trc}_{#1}}

\newcommand{\tLK}{L_{K(1)}\tK}

\let\iso\cong
\let\sma\wedge

\renewcommand{\to}{\mathchoice{\longrightarrow}{\rightarrow}{\rightarrow}{\rightarrow}}
\newcommand{\from}{\mathchoice{\longleftarrow}{\leftarrow}{\leftarrow}{\leftarrow}}
\newcommand{\sto}{\rightarrow}
\newcommand{\overto}[1]{\xrightarrow{\,#1\,}}

\newcommand{\phat}{^{\scriptscriptstyle\wedge}_{p}}

\let\catsymbfont\mathcal

\newcommand{\aK}{{\catsymbfont{K}}}

\newcommand{\bC}{{\mathbb{C}}}

\newcommand{\bCPi}{\bC P^{\infty}}
\newcommand{\bCPim}{\bC P^{\infty}_{-1}}
\newcommand{\bCPimbar}{\overline{\bC P}\mathstrut^{\infty}_{-1}}

\newcommand{\bF}{{\mathbb{F}}}

\newcommand{\bS}{{\mathbb{S}}}

\newcommand{\bZ}{{\mathbb{Z}}}
\newcommand{\bZp}{{\mathbb{Z}}\phat}
\newcommand{\bZpt}{({\mathbb{Z}}\phat)^{\times}}
\newcommand{\bQ}{{\mathbb{Q}}}
\newcommand{\bQp}{{\mathbb{Q}}\phat}

\newcommand{\oO}{\mathcal{O}}

\def\quickop#1{\expandafter\DeclareMathOperator\csname
#1\endcsname{#1}}
\quickop{id}\quickop{Id}
\quickop{Tor}\quickop{Ext}
\quickop{holim}\quickop{hocolim}\quickop{hofib}
\quickop{Tel}
\quickop{Pic}\quickop{Gal}
\quickop{tor}
\quickop{Hom}
\quickop{Diff}\quickop{spec}
\quickop{cont}\quickop{ab}\quickop{nr}
\quickop{Aut}

\newcommand{\ptor}{\tor_{p}}

\newtheorem{thm}{Theorem}
\numberwithin{thm}{section}

\newtheorem{cor}[thm]{Corollary}
\newtheorem{conj}[thm]{Conjecture}
\newtheorem{lem}[thm]{Lemma}
\theoremstyle{definition}
\newtheorem{prop}[thm]{Proposition}

\theoremstyle{remark}

\newcommand{\term}[1]{\textit{#1}}

\newdir{ >}{{}*!/-5pt/\dir{>}}

\makeatletter\let\c@equation\c@thm\makeatother

\renewcommand{\labelenumi}{(\theenumi)}

\ifx\texorpdfstring\undefined\def\texorpdfstring#1#2{#1}\fi

\begin{document}

\title[The homotopy groups of $K(\bS)$]
{The homotopy groups of the algebraic $K$-theory of the sphere spectrum%
\footnotemark}

\author{Andrew J. Blumberg}
\address{Department of Mathematics, The University of Texas,
Austin, TX \ 78712}
\email{blumberg@math.utexas.edu}
\author{Michael A. Mandell}
\address{Department of Mathematics, Indiana University,
Bloomington, IN \ 47405}
\email{mmandell@indiana.edu}



\begin{abstract}
We calculate $\pi_{*}K(\bS)\otimes \bZ[1/2]$, the homotopy groups of
$K(\bS)$ away from $2$, in terms of the homotopy 
groups of $K(\bZ)$, the homotopy groups of $\bCPim$, and the homotopy
groups of $\bS$.  This builds on work of Waldhausen, who
computed the rational homotopy groups (building on work of Quillen and
Borel) and Rognes, who calculated the groups at odd regular
primes in terms of the homotopy groups of $\bCPim$, and the homotopy
groups of $\bS$.
\end{abstract}

\maketitle

\footnotetext
{This is substantially the same as the work with identical title
and authors first published in \textit{Geometry \& Topology}~23(1) (2019),
published by Mathematical Sciences Publishers.\par \copyright\ 2019
Mathematical Sciences Publishers.}

\section{Introduction}

The algebraic $K$-theory of the sphere spectrum $K(\bS)$ is
Waldhausen's $A(*)$, the algebraic $K$-theory of the one-point
space.   The underlying infinite loop space of $K(\bS)$ splits as a
copy of the underlying infinite loop space of $\bS$ and the smooth
Whitehead space of a point $\whd(*)$.  For a 
highly-connected compact manifold $M$, the second loop space of $\whd(*)$ 
approximates the stable concordance space of $M$,
and the loop space of $\whd(*)$ parametrizes stable
$h$-cobordisms~\cite{WaldhausenJahrenRognes}.  As a consequence, computation of
the algebraic $K$-theory of the sphere spectrum is a fundamental
problem in algebraic and differential topology.

Early efforts in this direction were carried out by Waldhausen in
the 1980s using the ``linearization'' map $K(\bS) \to K(\bZ)$ from
the $K$-theory of the sphere spectrum to the $K$-theory of the
integers.  This is induced by the map of (highly structured) ring
spectra $\bS \to \bZ$.  Waldhausen showed that because the map $\bS\to
\bZ$ is an isomorphism on homotopy groups in degree zero (and below)
and a rational equivalence on higher homotopy groups, the map $K(\bS)\to
K(\bZ)$ is a rational equivalence.  Borel's computation of the
rational homotopy groups of $K(\bZ)$ then applies to calculate the
rational homotopy groups of $K(\bS)$: $\pi_{q}K(\bZ)\otimes \bQ$ is
rank 1 when $q=0$ or $q=4k+1$ for $k>1$ and is zero for all other $q$.

In the 1990s, work of B\"okstedt, Carlsson, Cohen, Goodwillie, Hsiang, and Madsen
revolutionized the computation of algebraic $K$-theory with the
introduction and study of \emph{topological cyclic homology} $TC$, an
analogue of negative cyclic homology that can be computed using the
methods of equivariant stable homotopy theory.  $TC$ is the target of
the \emph{cyclotomic trace}, a natural transformation $K\to TC$.  For
a map of highly structured ring spectra such as $\bS\to \bZ$,
naturality gives a diagram
\[
\xymatrix{
K(\bS) \ar[r] \ar[d] & TC(\bS) \ar[d] \\
K(\bZ) \ar[r] & TC(\bZ),
}
\]
the \emph{linearization/cyclotomic trace square}.
A foundational theorem of Dundas~\cite{Dundas-RelK} (building on work of
McCarthy~\cite{McCarthy} and Goodwillie~\cite{GoodwillieHN}) states
that the square above becomes homotopy
cartesian after $p$-completion, which means that the maps of homotopy
fibers become weak equivalences after $p$-completion.

In the 2000s, Rognes~\cite{Rognesp} used the
linearization/cyclotomic trace square to compute the homotopy groups
of $K(\bS)$ at odd regular primes in terms of the homotopy groups of $\bS$
and the homotopy groups of $\mathbb{C}P^{\infty}_{-1}$ (assuming the
now affirmed Quillen-Lichtenbaum conjecture).  The
answer is easiest to express in terms of the torsion subgroups:
Because $\pi_{n}K(\bS)$ is finitely generated, it is the direct sum of
a free part and a torsion part, the free part being $\bZ$ when $n=0$ or
$n\equiv 1\mod 4$, $n>1$, and $0$ otherwise.  The
main theorem of Rognes~\cite{Rognesp} is that for $p$ an odd regular prime
the $p$-torsion of $\pi_{*}K(\bS)$ is
\[
\ptor(\pi_{*}K(\bS))\iso 
\ptor(\pi_{*}\bS \oplus \pi_{*-1}c \oplus \pi_{*-1}\bCPimbar).
\]
(which can be made canonical, as discussed below).  Here $c$ denotes
the additive $p$-complete cokernel of $J$ spectrum (the connected
cover of the homotopy fiber of the map $\bS\phat\to L_{K(1)}\bS$); its
homotopy groups are all torsion and are direct summands of
$\pi_{*}\bS$.  The spectrum $\bCPimbar$ is a wedge summand of
$(\bCPim)\phat$~\cite[(1.3)]{MadsenSchlitkrull},
\cite[p.166]{Rognesp}, $(\bCPim)\phat\simeq \bCPimbar \vee \bS\phat$;
using unpublished work of Knapp, Rognes~\cite[4.7]{Rognesp} calculates
the order of these torsion groups in degrees $\leq 2(2p+1)(p-1)-4$.

The description of $\ptor(\pi_* K(\bS))$ above uses an
identification of the homotopy type of $TC(\bS)\phat$ as 
$\bS\phat \vee \Sigma\bCPim$ due to B\"okstedt-Hsiang-Madsen~\cite[5.15,5.17]{BHM}.
Unfortunately, there does not seem to be a canonical map realizing
this weak equivalence, and so the isomorphism above is not canonical.
We do have a canonical splitting $TC(\bS)\simeq \bS \vee \fu$ where
$\fu$ is the homotopy cofiber of the unit map $\bS\to TC(\bS)$, or
equivalently, the homotopy fiber of the map $TC(\bS)\to THH(\bS)$.
As we discuss below, the splitting of the $p$-completion $\fu\phat$,
\[
\fu\phat\simeq \fubar \vee \Sigma \bS\phat
\]
corresponding to the splitting $(\bCPim)\phat\simeq \bCPimbar\vee \bS\phat$ turns out to be canonical.  Rognes' canonical
identification of the $p$-torsion of $\pi_{*}K(\bS)$ is then
\[
\ptor(\pi_{*}K(\bS))\iso 
\ptor(\pi_{*}\bS \oplus \pi_{*-1}c \oplus \pi_{*}\fubar).
\]

This paper computes $\pi_* K(\bS)\phat$ in the case of irregular
primes, thereby completing the computation of the homotopy groups of
the algebraic $K$-theory of the sphere spectrum away from the prime $2$.  
As a first step, we
prove the following splitting theorem in Section~\ref{sec:proof}.

\begin{thm}\label{thm:mainles}
Let $p$ be an odd prime.  The long exact sequence on homotopy groups
induced by the $p$-completed linearization/cyclotomic trace square
breaks up into non-canonically split short exact sequences
\[
0\to \pi_{*}K(\bS)\phat\to \pi_{*}TC(\bS)\phat\oplus \pi_{*}K(\bZ)\phat\to
\pi_{*}TC(\bZ)\phat\to 0.
\]
\end{thm}

Choosing appropriate splittings in the previous theorem,
we can identify the $p$-torsion groups.  The identification is 
again in terms of $\pi_{*}\bS$ and $\pi_{*}\fu$ (or non-canonically,
$\pi_* \Sigma \bCPim$) but now involves also $\pi_{*}K(\bZ)$, which is not
fully understood at irregular primes.  In
the statement, $\tK(\bZ)$ denotes the wedge summand of
$K(\bZ)\phat$ complementary to $j$~\cite[2.1,9.7]{DwyerMitchell},
$K(\bZ)\phat\simeq j \vee \tK(\bZ)$, where $j$ is the $p$-complete
additive image of $J$ spectrum, the connective cover of $L_{K(1)}\bS$.
(As discussed below, this splitting is canonical.)

\begin{thm}\label{thm:main}
Let $p$ be an odd prime.  The $p$-torsion in $\pi_{*}K(\bS)$ admits
canonical isomorphisms
\begin{align}
\label{a}\tag{a}\ptor(\pi_{*}K(\bS))&
\iso \ptor(\pi_{*}c \oplus \pi_{*-1}c
\oplus \pi_{*}\fubar \oplus \pi_{*}K(\bZ))\\
\label{b}\tag{b}&\iso \ptor(\pi_{*}\bS \oplus \pi_{*-1}c
\oplus \pi_{*}\fubar \oplus \pi_{*}\tK(\bZ)).
\end{align}
\end{thm}

In Formula~\eqref{a}, the map $\ptor(\pi_{*}K(\bS))\to
\ptor(\pi_{*}(K(\bZ)))$ is induced by the linearization map and the map 
\[
\ptor(\pi_{*}K(\bS))\to \ptor(\pi_{*}c \oplus \pi_{*-1}c
\oplus \pi_{*}\fubar)
\]
is induced by the composite of the cyclotomic trace map $K(\bS)\to
TC(\bS)$ and a canonical splitting~\eqref{eq:TCSsplit} of the homotopy groups
$\pi_{*}TC(\bS)\phat$ as
\[
\pi_{*}TC(\bS)\phat \iso
\pi_{*}(j)\oplus \pi_{*}(c)\oplus \pi_{*}(\Sigma j)\oplus
\pi_{*}(\Sigma c) \oplus \pi_{*}(\fubar)
\]
explained in the first part of Section~\ref{sec:review}, followed by
the projection onto the non-$j$ summands.

In formula~\eqref{b}, the map $\ptor(\pi_{*}K(\bS))\to
\ptor(\pi_{*}\tK(\bZ))$ is induced by the linearization map and the
canonical map $\pi_{*}K(\bZ)\to \pi_{*}\tK(\bZ)$ which is the quotient
by the image of the unit map $\pi_{*}\bS\to \pi_{*}K(\bZ)$. The map 
\[
\ptor(\pi_{*}K(\bS))\to \ptor(\pi_{*}\bS \oplus \pi_{*-1}c
\oplus \pi_{*}\fubar)
\]
is induced by the composite of the cyclotomic trace map $K(\bS)\to
TC(\bS)$ and the canonical splitting of homotopy groups
\[
\pi_{*}TC(\bS)\phat \iso
\pi_{*}(\bS\phat)\oplus \pi_{*}(\Sigma j)\oplus
\pi_{*}(\Sigma c) \oplus \pi_{*}(\fubar)
\]
followed by projection away from the $\Sigma j$ summand.  (The splitting of
$\pi_{*}TC(\bS)\phat$ here is related to the splitting above by the
canonical splitting on homotopy groups $\pi_{*}\bS\phat\iso
\pi_{*}(j)\oplus \pi_*(c)$.)

Formula~\eqref{b} generalizes the computation of Rognes~\cite{Rognesp}
at odd regular primes because $\tK(\bZ)$ is torsion free if (and only
if) $p$ is regular (see for example \cite[\S VI.10]{Weibel-KBook}).
Part of the argument for the theorems above involves making certain
splittings in prior $K$-theory and $TC$ computations canonical and
canonically identifying certain maps (or at least their effect on
homotopy groups). Although we construct the splittings and prove their
essential uniqueness calculationally, we offer in
Section~\ref{sec:splitting} a theoretical explanation in terms of a
conjectural extension of Adams operations on algebraic $K$-theory 
to an action of the $p$-adic
units and a conjecture on the consistency of Adams operations on
$K$-theory and Adams operations on $TC$.  This perspective leads to a
splitting of $K(\bS)\phat$ and the linearization/cyclotomic trace
square into $p-1$ summands (which is independent of the conjectures),
expanding on certain splittings discovered by Rognes~\cite[\S3]{Rognesp}
(in the case of regular primes); see Theorem~\ref{thm:eigensquare}.

The identification of the maps in Section~\ref{sec:maps} allows us to
prove the following theorem, which slightly sharpens Theorem~\ref{thm:main}.

\begin{thm}\label{thm:subgp}
Let $p$ be an odd prime. Let $\alpha \colon \pi_{*}TC(\bS)\to
\pi_{*}j$ be the induced map on homotopy groups given by the composite
of the canonical maps $TC(\bS)\to THH(\bS)\simeq \bS$ and $\bS\to j$;
let $\beta \colon \pi_{*}K(\bZ)\to \pi_{*}j$ be the canonical
splitting; and let $\gamma\colon
\pi_{*}TC(\bS)\to \pi_{*}\Sigma \bS^{1}\to \pi_{*}\Sigma j$ be the map
induced by the splitting of~\eqref{eq:TCSsplit} below.  Then
the $p$-torsion subgroup of $\pi_{*}K(\bS)$ maps isomorphically to the
subgroup of the $p$-torsion subgroup of $\pi_{*}TC(\bS)\oplus
\pi_{*}K(\bZ)$ where the appropriate projections composed with
$\alpha$ and $\beta$ agree and the appropriate projection composed
with $\gamma$ is zero.
\end{thm}

Theorems~\ref{thm:main} and~\ref{thm:subgp} provide a good
understanding of what the linearization/cyclotomic trace square does
on the odd torsion part of the homotopy groups.  In contrast, the maps
on mod torsion homotopy groups
are not fully understood.  We have that $\pi_{n}K(\bS)$ and
$\pi_{n}K(\bZ)$ mod torsion are rank one for $n=0$ and $n\equiv 1
\mod{4}$, $n>1$; the map $K(\bS)\to K(\bZ)$ is a rational equivalence
and an isomorphism in degree zero, but is not an isomorphism on mod
torsion homotopy groups in degrees congruent to $1$ mod $2(p-1)$ by
the work of Klein-Rognes (see the proof of
\cite[6.3.(i)]{KleinRognes}).  The mod torsion homotopy groups of
$TC(\bS)\phat$ are rank one in degree zero and odd degrees $\geq -1$;
the map $K(\bS)\phat\to TC(\bS)\phat$ on mod torsion homotopy groups
is an isomorphism in degree zero, by necessity zero in degrees not
congruent to $1$ mod $4$, and for odd regular primes an isomorphism in
degrees congruent to $1$ mod $4$.  For irregular primes, the map is
not fully understood.

In principle, we can use Theorem~\ref{thm:main} to calculate
$\pi_{*}K(\bS)$ in low degrees.  In practice, we are limited by a lack
of understanding of $\pi_{*} \bCPim$ and $\pi_* K(\bZ)$; we know
$\pi_{*}\bS$ and $\pi_{*}c$ in a comparatively larger range.  The
calculations of $\pi_* \bCPimbar$ for $*< |\beta_{2}|-1=(2p+1)(2p-2) - 3$
in \cite[4.7]{Rognesp} work for irregular primes as well as odd
regular primes; although they only give the order of the torsion rather
than the torsion group, in many cases the order is either $1$ or $p$,
and the group structure is also determined.  At primes that satisfy
the Kummer-Vandiver
conjecture, we know $\pi_{*}K(\bZ)$ in terms of Bernoulli numbers; at
other primes we know $\pi_{*}K(\bZ)$ in odd degrees, but do not
know $K_{4n}\bZ$ at all (except $n=0,1$) and only know the order of
$\pi_{4n+2}K(\bZ)$ and again only in terms of Bernoulli numbers.  As
the formula~\cite[4.7]{Rognesp} for $\pi_{*}\bCPimbar$ is somewhat messy,
we do not summarize the answer here, but for convenience, we have
included it in Section~\ref{sec:mess}.

As a consequence of the
work of Rognes~\cite[3.6,3.8]{Rognesp}, the cyclotomic trace 
\[
\trc{p}\colon K(\bS)\phat\to TC(\bS)\phat
\]
is injective on homotopy groups at
odd regular primes.  In Theorem~\ref{thm:main} above, the contribution
in $\ptor(\pi_{*}K(\bS))$ of $p$-torsion from $\ptor(\pi_{*}\tK(\bZ))$
maps to zero in 
$\pi_{*}TC(\bS)$ under the cyclotomic trace.  This then gives the
following complete answer to the question the authors posed
in~\cite{BM-Anil}:

\begin{cor}\label{cor:srAnil}
For an odd prime $p$, the cyclotomic trace $\trc{p}\colon
K(\bS)\phat\to TC(\bS)\phat$ is injective on homotopy groups if and
only if $p$ is regular. 
\end{cor}

Theorem~\ref{thm:mainles} contains the following more general injectivity result.

\begin{cor}
For an odd prime $p$, the map of ring spectra
\[
K(\bS)\phat \to TC(\bS)\phat \times K(\bZ)\phat
\]
is injective on homotopy groups.
\end{cor}

\subsection*{Conventions}
Throughout this paper $p$ denotes an odd prime.  For a ring $R$, we
write $R^{\times}$ for its group of units.

\subsection*{Acknowledgments}
The authors especially thank Lars Hesselholt for many helpful
conversations and for comments on a previous draft that led to
substantial sharpening of the results.  We are grateful to an
anonymous referee who wrote an extremely careful and detailed report,
which led to many improvements and corrections.
The authors also thank Mark
Behrens, Bill Dwyer, Tom Goodwillie, John Greenlees, Mike Hopkins, Thomas Kragh, 
Tyler Lawson, and Birgit Richter for helpful conversations, John
Rognes for useful comments, and the IMA and MSRI for their hospitality
while some of this work was done.

The first author was supported in part by NSF grant
DMS-1151577; the second author was supported in part by NSF grants
DMS-1105255, DMS-1505579.

\section{The spectra in the linearization/cyclotomic trace square}\label{sec:review}

We begin by reviewing the descriptions of the spectra $TC(\bS)\phat$,
$K(\bZ)\phat$, and $TC(\bZ)\phat$ in the $p$-completed
linearization/cyclotomic trace square  
\[
\xymatrix{
K(\bS)\phat \ar[r] \ar[d] & TC(\bS)\phat \ar[d] \\
K(\bZ)\phat \ar[r] & TC(\bZ)\phat.
}
\]
These three spectra have been identified in more familiar terms up to
weak equivalence.  We discuss how canonical these weak equivalences
are.  Often this will
involve studying splittings of the form $X\simeq Y\vee Z$ (or with
additional summands).  We will say that the \term{splitting is canonical} 
when we have a canonical isomorphism in the stable category $X\simeq
X'\vee X''$ with $X'$ and $X''$ (possibly non-canonically) isomorphic
in the stable category to $Y$ and $Z$; we will say that \term{the
identification of the summand $Y$ is canonical} when further the isomorphism in
the stable category $X'\simeq Y$ is canonical.  

To illustrate the above terminology, and justify its utility, consider
the example when $X$ is non-canonically weakly equivalent to $Y\vee Z$
and $[Y,Z]=0=[Z,Y]$ (where $[-,-]$ denotes maps in the stable
category).  In the terminology above, this gives an example of a
canonical splitting $X\simeq Y\vee Z$ without canonical identification
of the summands.  As another example, if we have a canonical map $Y\to
X$ in the stable category and a canonical map $X\to Y$ in the stable
category giving a retraction, then we have a
canonical splitting with canonical identification of summands $X\simeq
Y \vee F$ where $F$ is the homotopy fiber of the retraction map $X\to
Y$.  In this second example if we also have a non-canonical weak
equivalence $Z\to F$, we then have a canonical splitting $X\simeq
Y\vee Z$ with a canonical identification of the summand $Y$ (but not
the summand $Z$).

\subsection*{The splitting of \texorpdfstring{$TC(\bS)\phat$}{TC(S)p}}

Historically, $TC(\bS)\phat$ was the first of the terms in the
linearization/cyclotomic trace square to be understood.
Work of B\"okstedt-Hsiang-Madsen~\cite[5.15,5.17]{BHM} identifies the
homotopy type of
$TC(\bS)\phat$ as 
\[
TC(\bS)\phat\simeq \bS\phat\vee \Sigma (\bCPim)\phat.
\]
The inclusion of the $\bS$ summand is the unit of the ring spectrum
structure and is split by the canonical map $TC(\bS)\phat\to
THH(\bS)\phat$ and canonical identification $\bS\phat\simeq THH(\bS)\phat$
(also induced by the ring spectrum structure).  We therefore get a
canonical isomorphism in the stable category between $TC(\bS)\phat$ and
$\bS\phat\vee \fu\phat$ where $\fu\phat$ is the homotopy cofiber of
the map $\bS\phat \to TC(\bS)\phat$; in the terminology at the beginning of the section,
this is a canonical splitting with canonical identification
\begin{equation}\label{eq:TCS}
TC(\bS)\phat\simeq \bS\phat \vee \fu\phat.
\end{equation}
We have
a canonical splitting with non-canonical identification
$TC(\bS)\phat\simeq \bS\phat \vee \Sigma (\bCPim)\phat$.

As already indicated, we make use of a further splitting
from~\cite[\S 3]{Rognesp} (see also the remarks
preceding~\cite[(1.3)]{MadsenSchlitkrull}),
\[
(\bCPim)\phat\simeq \bS\phat \vee \bCPimbar,
\]
induced by the splitting $\Sigma^{\infty}_{+}\bCPi\simeq
\Sigma^{\infty}\bCPi\vee \bS$.
The splitting exists after inverting $2$, but for notational
convenience, we use it only after $p$-completion.  It follows that
there exists an analogous splitting 
\begin{equation}\label{eq:fubar}
\fu\phat \simeq \Sigma \bS\phat \vee \fubar.
\end{equation}
There is a canonical map $\Sigma \bS\phat\to
TC(\bS)\phat$ in the stable category arising from
\cite[5.15--17]{BHM}, which in particular 
gives a canonical isomorphism 
\[
\pi_{1}(TC(\bS)\phat)\iso \lim \bZ/p^{n}.
\]
Ravenel's calculation~\cite[1.11.(iii)]{Ravenel-SegalConsequences} 
reveals that $[\bCPim,\bS\phat]\iso \bZ\phat$, so there is a unique
map $\fu\phat\to \Sigma \bS\phat$ in the stable category such that the
composite self-map of $\Sigma \bS\phat$ is the identity. 

Finally, we have a canonical isomorphism of homotopy groups
$\pi_{*}\bS\phat\iso \pi_{*}(j)\oplus \pi_{*}(c)$ from classical work in
homotopy theory on Whitehead's $J$-homomorphism and Bousfield's work
on localization of spectra~\cite[\S 4]{BousfieldLocalizationSpectra}.
As above, $j$ denotes the connective cover of the
$K(1)$-localization of the
sphere spectrum and $c$ denotes the 
homotopy fiber of the map $\bS\phat\to j$.  The map $\bS\phat \to j$
induces an isomorphism from the $p$-Sylow subgroup of the image-of-$J$
subgroup of $\pi_{q}\bS\phat$ to $\pi_{q}j$.

Putting this all together, we have a canonical isomorphism
\begin{equation}\label{eq:TCSsplit}
\pi_{*}TC(\bS)\phat \iso
\pi_{*}(j)\oplus \pi_{*}(c)\oplus \pi_{*}(\Sigma j)\oplus
\pi_{*}(\Sigma c) \oplus \pi_{*}(\fubar).
\end{equation}

\subsection*{The splitting of \texorpdfstring{$TC(\bZ)\phat$}{TC(Z)p}}

Next up historically is $TC(\bZ)$, which was first identified by
B\"okstedt-Madsen~\cite{BokstedtMadsen, BokstedtMadsen2} and 
Rognes~\cite{Rognes-Characterizing}.  They expressed the
answer on the infinite-loop space level and (equivalently) described
the connective cover spectrum $TC(\bZ)\phat[0,\infty)$ as having the
homotopy type of 
\[
j\vee \Sigma j \vee \Sigma^{3}ku\phat
\simeq j\vee \Sigma j\vee \Sigma^{3}\ell\vee \Sigma^{5}\ell \vee
\dotsb \vee \Sigma^{3+2(p-2)}\ell
\]
(non-canonical isomorphism in the stable category).
Here $ku$ denotes connective complex topological $K$-theory (the
connective cover of periodic complex topological $K$-theory $KU$), and
$\ell$ denotes the Adams summand of $ku\phat$ (the connective cover
of the Adams summand $L$ of $KU\phat$).  
A standard calculation (e.g.,
see~\cite[2.5.7]{Madsen-Traces}) shows that before taking the
connective cover $\pi_{-1}(TC(\bZ)\phat)$ is free of rank one over
$\bZp$, and as we explain in the subsection on the homotopy type of
$TC(\bZ)\phat$ below, the argument
for~\cite[3.3]{Rognesp} extends to show that
summand $\Sigma^{3+2(p-3)}\ell$ above becomes $\Sigma^{-1}\ell$ in
$TC(\bZ)\phat$.

The non-canonical splitting above rigidifies into a canonical
splitting and canonical identification
\begin{equation}\label{eq:TCZmain}
TC(\bZ)\phat\simeq j \vee \Sigma j' \vee \Sigma^{-1}\ell_{TC}(0) \vee
\Sigma^{-1}\ell_{TC}(p)\vee \Sigma^{-1}\ell_{TC}(2)\vee \cdots \vee \Sigma^{-1}\ell_{TC}(p-2).
\end{equation}
Here the numbering replaces $1$ with $p$ but otherwise numbers
sequentially\break $0$,\dots, $p-2$.  Each $\Sigma^{-1}\ell_{TC}(i)$ is a
spectrum that is non-canonically weakly equivalent to
$\Sigma^{2i-1}\ell\iso\Sigma^{2i}(\Sigma^{-1}\ell)$ but that admits a
canonical description by work of
Hesselholt-Madsen~\cite[Th.~D]{HM2} and
Dwyer-Mitchell~\cite[\S13]{DwyerMitchell} in terms of units of
cyclotomic extensions of $\bQp$; we omit a detailed description of the
identification as we do not use it.
The spectrum $j'$ is the connective cover of the $K(1)$-localization
of the $\bZpt$-Moore spectrum $M_{\bZpt}$, or equivalently, the
$p$-completion of $j\sma M_{\bZpt}$.  Since the $p$-completion of
$\bZpt$ is non-canonically isomorphic to $\bZp$, we have that $j'$ is
non-canonically weakly equivalent to $j$ but with $\pi_{0}j'$
canonically isomorphic to $(\bZpt)\phat$ (and isomorphisms in the
stable category from $j$ to $j'$ are in canonical bijection with
isomorphisms $\bZp\to (\bZpt)\phat$).  Much of the canonical splitting
in~\eqref{eq:TCZmain} follows by a calculation of maps in the stable
category.  Temporarily writing $x(0)=j\vee \Sigma^{-1}\ell$,
$x(1)=\Sigma j\vee \Sigma^{2p-1} \ell$, and $x(i)=\Sigma^{2i-1}\ell$
for $i=2,\ldots,p-2$, the results above give us a canonical splitting
without canonical identification (see the start of this section for an
explanation of this terminology)
\[
TC(\bZ)\phat \simeq x(0) \vee \cdots \vee x(p-2)
\]
because for $i\neq i'$, we have $[x(i),x(i')]=0$; we provide a
detailed computation as Proposition~\ref{prop:TCsummands} at the end
of the section. The canonical map $\bS\to j$ induces an isomorphism
$[j,TC(\bZ)\phat]\iso \pi_{0}(TC(\bZ)\phat)$ and so we have a canonical
map $\eta \colon j\to TC(\bZ)\phat$ coming from the unit of the ring spectrum
structure.  Likewise, the canonical map $M_{\bZpt}\to j'$ induces an
isomorphism  
\[
[\Sigma j',TC(\bZ)\phat]\iso [\Sigma M_{\bZpt},TC(\bZ)\phat]\iso \Hom((\bZpt)\phat,\pi_{1}TC(\bZ)\phat).
\]
The canonical isomorphism $\pi_{1}TC(\bZ)\phat\iso (\bZpt)\phat$
\cite[Th.~D]{HM2}, then gives a canonical map $u\colon \Sigma j'\to
TC(\bZ)\phat$.  The restriction along $\eta$ and $u$ induce bijections
\[
[TC(\bZ)\phat,j]\iso [j,j]\qquad \text{and}\qquad 
[TC(\bZ)\phat,\Sigma j']\iso [\Sigma j',\Sigma j']
\]
(q.v.~Proposition~\ref{prop:splitjp} below), giving retractions to
$\eta$ and $u$ that are unique in the stable category.  This gives a canonical
splitting and identification of the $j$ and $\Sigma j'$ summands
in~\eqref{eq:TCZmain}.

\subsection*{The splitting of \texorpdfstring{$K(\bZ)\phat$}{K(Z)p}}

The homotopy type of $K(\bZ)\phat$ is still not fully
understood at irregular primes even in light of the confirmation of
the Quillen-Lichtenbaum conjecture.  The Quillen-Lichtenbaum
conjecture implies that $K(\bZ)\phat$ can be understood in terms of
its $K(1)$-localization $L_{K(1)}K(\bZ)$, and work of
Dwyer-Friedlander~\cite{DwyerFriedlander-TopologicalModels} or
Dwyer-Mitchell~\cite[12.2]{DwyerMitchell} identifies the homotopy type
of $K(\bZ)\phat$ at regular odd primes as $j\vee \Sigma^{5}ko\phat$
(non-canonically).  At any
prime, Quillen's Brauer induction and reduction mod $r$ (for $r$
prime a generator of $(\bZ/p^{2})^{\times}$) induce a splitting
\begin{equation}\label{eq:KZ}
K(\bZ)\phat\simeq j\vee \tK(\bZ)
\end{equation}
for some $p$-complete spectrum we denote as $\tK(\bZ)$ (see for
example, \cite[2.1,5.4,9.7]{DwyerMitchell}).  We argue in
Proposition~\ref{prop:jtKZ} at the end of this section
(once we have reviewed more about $\tK(\bZ)$) that $[j,\tK(\bZ)]=0$
and $[\tK(\bZ),j]=0$, and it follows that the splitting
in~\eqref{eq:KZ} is canonical; the identification
of the summand $j$ is also canonical, as the map $j\to K(\bZ)$ described is then
the unique one taking the canonical generator of $\pi_{0}j$ to the
unit element of $\pi_{0}K(\bZ)$ in the ring spectrum structure.

The splitting $K(\bZ)\phat\simeq
j\vee \tK(\bZ)$ corresponds to a splitting 
\[
L_{K(1)}K(\bZ)\simeq J\vee \tLK(\bZ)
\]
where $J=L_{K(1)}\bS$ is the $K(1)$-localization of $\bS$.
We have that $\tK(\bZ)$ is 
$4$-connected \cite{LeeSzczarba76,LeeSzczarba78,Rognes-K4Z} and the Quillen-Lichtenbaum conjecture (as reformulated
by Waldhausen~\cite[\S4]{Waldhausen-Chromatic}) implies that
$K(\bZ)\phat\to L_{K(1)}K(\bZ)$ induces an isomorphism on homotopy groups in
degrees $2$ and above.  Thus, $\tK(\bZ)$ is the $4$-connected cover of
$\tLK(\bZ)$.  This makes it straightforward to convert
statements about the homotopy type of $L_{K(1)}\tK(\bZ)$ into statements
about the homotopy type of $\tK(\bZ)\phat$.

A lot about $L_{K(1)}K(\bZ)$ is known by work of
Dwyer-Mitchell~\cite{DwyerMitchell}, which studies $L_{K(1)}K(R_{F})$
where $R_{F}=\oO_{F}[1/p]$ is the ring of $p$-integers in a number
field; in the case $F=\bQ$, $R_{F}=\bZ[1/p]$.  By Quillen's
localization theorem (and the triviality of $L_{K(1)}K(\bF_{p})\simeq
L_{K(1)}H\bZp$),
\[
L_{K(1)}K(\bZ)\simeq L_{K(1)}K(\bZ[1/p])
\]
and \cite[1.7]{DwyerMitchell} in particular computes
$(KU\phat)^{*}(\tK(\bZ))$ together with the action of
$(KU\phat)^{0}(KU\phat)$ in number theoretic terms in terms of the
\textit{Iwasawa module} $M$.  
Let $\mu_{p^{n}}$ denote the $p^{n}$th roots of unity and let
$\mu_{p^{\infty}}=\bigcup \mu_{p^{n}}$.  The Iwasawa module $M$ is
then the Galois group of the maximal abelian
$p$-extension of $\bQ(\mu_{p^{\infty}})$ unramified except at $p$.  It
comes with an action of $\Gamma'=\Gal(\bQ(\mu_{p^{\infty}})/\bQ)$ and its
completed group ring, typically denoted as $\Lambda'$ in Iwasawa
theory.  The canonical isomorphism $\Gamma'\iso (\bZp)^{\times}$
(induced by the canonical isomorphisms $\Aut(\mu_{p^{n}})\iso
(\bZ/p^{n})^{\times}$) and the interpolation of Adams operations on $KU\phat$ to
$p$-adic units gives an isomorphism between $\Gamma'$ and this group of
$p$-adically interpolated Adams operations, and induces an isomorphism of
$\bZp$-algebras $\Lambda'\iso (KU\phat)^{0}(KU\phat)$.  
From here on, we identify these two algebras via this isomorphism, the
utility of which is explained in~\cite[\S4,\S6]{DwyerMitchell}.
The main result of
Dwyer-Mitchell~\cite[1.7]{DwyerMitchell} proves 
\begin{align*}
(KU\phat)^{0}(\tK(\bZ))&=0\\
(KU\phat)^{-1}(\tK(\bZ))&\iso M
\end{align*}
as $\Lambda'$-algebras.  As discussed in~\cite[\S8]{DwyerMitchell},
$M$ is finitely generated projective dimension one as a
$\Lambda'$-module.  Working in terms of $L$, the Adams summand of $KU\phat$,
it follows that $L^{*}(\tK(\bZ))$ is concentrated in
odd degrees where it is a finitely generated projective dimension one
$L^{0}L$-module in each degree.  In particular, $\tLK(\bZ)$ splits 
\[
\tLK(\bZ)\simeq Y_{0}\vee \dotsb \vee Y_{p-2}
\]
where $Y_{i}$ is the fiber of a map
from a finite wedge of copies of $\Sigma^{2i-1}L$ to a finite wedge of
copies of $\Sigma^{2i-1}L$, giving an $L^{0}L$-resolution of the module
$L^{2i-1}Y_{i}$.  We note that $[Y_{i},Y_{i'}]=0$ unless $i=i'$.  

Letting $y_{i}$ be the $4$-connected cover of $Y_{i}$, we show below
in Proposition~\ref{prop:KZsummands} that $[y_{i},y_{i'}]=0$ for
$i\neq i'$ and so obtain a canonical splitting and canonical
identification of summands
\begin{equation}\label{eq:KZmain}
K(\bZ)\phat\simeq j\vee y_{0}\vee \cdots \vee y_{p-2}.
\end{equation}

Dwyer-Mitchell~\cite[8.10]{DwyerMitchell} relates the $L^{0}L$-modules
$L^{2i-1}Y_{i}$ to class groups of cyclotomic fields.  Write $A_{m}$
for the $p$-Sylow subgroup of the ideal class group of the cyclotomic
field $\bQ(\mu_{p^{m+1}})$, and let $A_{\infty}$ denote the inverse
limit of $A_{m}$ over the norm maps, with its natural structure of a
$\Lambda'$-module.  Usual notation is to write $\Gamma$ for the
subgroup of $\Gamma'$ given by
$\Gal(\bQ(\mu_{p^{\infty}})/\bQ(\mu_{p}))$, or equivalently, the
subgroup of $(\bZp)^{\times}$ of $p$-adic units that are congruent to
$1$ mod $p$.  Writing $\Delta$ for $\Gal(\bQ(\mu_{p})/\bQ)$, or
equivalently, $(\bZ/p)^{\times}$, the Teichm\"uller character
$\omega\colon \Delta\to \bZpt$ then induces a splitting
$\Gamma'\iso \Gamma \times \Delta$, which in turn induces an
isomorphism $\Lambda'=\Lambda[\Delta]$, for a certain
subring $\Lambda$ of $\Lambda'$.  In these terms,
\cite[8.10]{DwyerMitchell} gives an exact sequence
\begin{multline}\label{eq:dmmain}
0\to \Ext^{1}_{\Lambda}(A_{\infty},\Lambda)\to (KU\phat)^{1}(\tK(\bZ))\to
\Hom_{\Lambda}(E'_{\infty}(\text{red}),\Lambda)\to\\
\Ext^{2}_{\Lambda}(A_{\infty},\Lambda)\to 0.
\end{multline}
(Note that in the case under consideration here, $A_{\infty}$ is also
isomorphic to the $\Lambda'$-modules denoted $L_{\infty}$
and $A'_{\infty}$ in \cite{DwyerMitchell}, q.v.~\S12.)  Here
$E'_{\infty}(\text{red})$ is a certain $\Lambda'$-module defined in
terms of cyclotomic units, the details of which will not come into
play here, except to note that by \cite[9.10]{DwyerMitchell},
$E'_{\infty}(\text{red})$ is non-canonically isomorphic to
$(KU\phat)^{0}(KO\phat)$ as a $\Lambda'$-module.

To decompose~\eqref{eq:dmmain} in terms of the $Y_{i}$, we employ
the ``Adams splitting'' or ``eigensplitting'' of $\Delta$-equivariant
$\bZp$-modules.  Any $\bZp[\Delta]$-module $X$ decomposes as a direct
sum of pieces corresponding to the
powers $\omega^{j}$ of the Teichm\"uller character $\omega\colon
\Delta \to \bZpt$. The $\omega^{j}$-character piece $\epsilon_{j}X$ is
the submodule where the element $\alpha$ of $\Delta$ acts by multiplication by
$\omega^{j}(\alpha)\in \bZp$ (for all $\alpha\in \Delta$).  (For
$\Lambda'=\Lambda[\Delta]$-modules, the $\omega^{j}$ character piece
is the summand where $\psi^{\omega(\alpha)}$ acts by
$\omega^{j}(\alpha)$ for all $\alpha \in \Delta$.)  This
relates to the Adams decomposition of $KU\phat$ as
\[
\epsilon_{j}((KU\phat)^{q}(Z))\iso L^{2j+q}(Z)
\]
for any spectrum $Z$, and 
\[
(\Sigma^{2j}L)^{0}(\Sigma^{2j}L)\overto{\iso}(KU\phat)^{0}(\Sigma^{2j}L)\to (KU\phat)^{0}(KU\phat)=\Lambda'
\]
is reasonably interpreted as the inclusion of $\epsilon_{j}\Lambda'$.
(The projection $(KU\phat)^{0}(KU\phat)\to
(\Sigma^{2j}L)^{0}(\Sigma^{2j}L)$ induces an isomorphism of rings from
$\Lambda$ to $(\Sigma^{2j}L)^{0}(\Sigma^{2j}L)$.)
In $\Delta$-equivariant terms, the exact sequence~\eqref{eq:dmmain} is better
written as
\begin{multline*}
0\to \Ext^{1}_{\Lambda}(A_{\infty},L^{0}L)\to (KU\phat)^{1}(\tK(\bZ))\to
\Hom_{\Lambda}(E'_{\infty}(\text{red}),L^{0}L)\to\\
\Ext^{2}_{\Lambda}(A_{\infty},L^{0}L)\to 0.
\end{multline*}
(q.v.~\cite[8.11]{DwyerMitchell}) with the canonical 
action of $\Delta$ on $\Hom$ and $\Ext^{i}$.  Taking the
$\epsilon_{j}$ piece of the
$\Delta$-eigensplitting of this sequence, we get exact sequences
\begin{multline*}
0\to \Ext^{1}_{\Lambda}(\epsilon_{-j}A_{\infty},L^{0}L)\to
L^{2j+1}Y_{j+1}\to
\Hom_{\Lambda}(\epsilon_{-j}E^{\prime}_{\infty}(\text{red}),L^{0}L)\to \\
\Ext^{2}_{\Lambda}(\epsilon_{-j}A_{\infty},L^{0}L)\to 0
\end{multline*}
since $L^{2j+1}Y_{k}=0$ for $k\neq j+1$ (where we understand
$Y_{p-1}=Y_{0}$).  Now
$\Hom_{\Lambda}(\epsilon_{-j}E^{\prime}_{\infty}(\text{red}),L^{0}L)$
is zero when $j$ is odd and a free $\Lambda$-module of rank 1 when $j$
is even.  Specifically, $Y_{i}$ is closely related to
$\epsilon_{j}A_{\infty}$ for $i+j\equiv 1\mod{(p-1)}$. 

For fixed $p$, several of the $\omega^{j}$-character pieces of
$A_{\infty}$ are always zero.  In
particular, $\epsilon_{0}A_{0}=0$ (because
it is canonically isomorphic to the $p$-Sylow subgroup of the projective class
group of $\bZ$) and \cite[10.7]{Washington-1997} then implies that
$\epsilon_{0}A_{\infty}=0$.  From the exact sequence above,
$L^{1}Y_{1}\iso 
L^{0}L$ (non-canonically)
and so $Y_{1}$ is non-canonically weakly equivalent to $\Sigma L$. It
follows that $y_{1}$ is non-canonically weakly equivalent to
$\Sigma^{1+2(p-1)}\ell$.  In terms of
$K(\bZ)\phat$, we obtain a further canonical splitting (without
canonical identification)
$\tK(\bZ)\simeq \Sigma^{2p-1}\ell \vee \tK\mathstrut^{\#}(\bZ)$
for some $p$-complete spectrum $\tK\mathstrut^{\#}(\bZ)$.   We use the
identification of $y_{1}$ as a key step in the proof of
Theorem~\ref{thm:mainles} in Section~\ref{sec:proof}.

Another useful vanishing result is $\epsilon_{1}A_{0}=0$
\cite[6.16]{Washington-1997}.  As a consequence, we see that
\begin{equation}\label{eq:kvanish}
Y_{0}\simeq *.
\end{equation}
This simplifies some formulas and arguments.

Although these are the only results we use, other vanishing results
for $\epsilon_{j}A$ give other vanishing results for the summands $Y_{i}$.
Herbrand's Theorem~\cite[6.17]{Washington-1997} and Ribet's
Converse~\cite{Ribet-ConverseHerbrant}, \cite[15.8]{Washington-1997}
state that for $3\leq j\leq p-2$ 
odd, $\epsilon_{j}A_{0}\neq 0$ if and only if $p|B_{p-j}$ where $B_{n}$
denotes the Bernoulli number, numbered by the convention
$\frac{t}{e^{t}-1}=\sum B_{n}\frac{t^{n}}{n!}$.  We see that for $p-3 \geq 
i\geq  2$ even, $Y_{i}\simeq *$ when $p$ does not divide $B_{i+1}$.  In
particular, $Y_{2}$, $Y_{4}$, $Y_{6}$, $Y_{8}$, and $Y_{10}$ are
trivial, $Y_{12}$ is trivial for $p\neq 691$, and for every even $i$,
$Y_{i}$ is only nontrivial for finitely many primes.

A prime $p$ is regular precisely when $p$ does not divide the class
number of $\bQ(\zeta_{p})$, or in other words, when $A_{0}=0$ and
therefore $A_{\infty}=0$.  Then for an odd regular prime, we have that
$Y_{2k}$ is trivial for all $k$ and 
$Y_{2k+1}$ is non-canonically weakly equivalent to $\Sigma^{4k+1}L$.  It follows that $\tLK(\bZ)$
is non-canonically weakly equivalent to $\Sigma KO\phat$ and
$\tK(\bZ)$ is non-canonically weakly equivalent to $\Sigma^{5} ko\phat$, since $\tK(\bZ)$ is the
$4$-connected cover of $\tLK(\bZ)$.  This leads precisely to the
description of $K(\bZ)\phat$ as non-canonically weakly equivalent to
$j\vee \Sigma^{5}ko\phat$, as indicated above. 

Now consider the case when $p$ satisfies the Kummer-Vandiver
condition.  This means that $p$
does not divide the class number of 
$\bQ(\zeta_{p}+\zeta_{p}^{-1})$ (the fixed field of $\bQ(\zeta_{p})$
under complex conjugation).  The $p$-Sylow subgroup 
is precisely the subgroup of $A_{0}$ fixed by complex conjugation,
which is the internal direct sum of $\epsilon_{j}A_{0}$ for $0\leq
j<p-1$ even.  It follows that $\epsilon_{j}A_{0}=0$ for $j$ even, and
so again $Y_{2k+1}$ is non-canonically weakly equivalent to
$\Sigma^{4k+1}L$; this splits a copy of $\Sigma KO\phat$ (with non-canonical
identification) off $\tLK$ as in~\cite[12.2]{DwyerMitchell}.  Now the even summands 
$Y_{2k}$ may be non-zero, but 
the $\Lambda$-modules $\epsilon_{j}A_{\infty}$ are cyclic for $j$
odd (see for example, \cite[10.16]{Washington-1997}) and
$Y_{2k}$ is (non-canonically) weakly equivalent to the homotopy
fiber of a map 
$\Sigma^{4k-1}L\to \Sigma^{4k-1}L$ determined by the $p$-adic $L$-function
$L_{p}(s;\omega^{2k})$ \cite[12.2]{DwyerMitchell}.  As above,
$Y_{0}\simeq *$ and in the other cases, for $n>0$, $n\equiv
2k-1\mod {(p-1)}$,
\begin{align*}
&\pi_{2n}Y_{2k}
  \iso\bZp/L_{p}(-n,\omega^{2k})
  =\bZp/\biggl(\frac{\scriptstyle B_{n+1}}{\scriptstyle n+1}\biggr)\\
&\pi_{2n+1}Y_{2k}=0
\end{align*}
(non-canonical isomorphisms).  The groups are of course zero for
$n\not\equiv 2k-1\mod {(p-1)}$.  (For $n<0$, $n\equiv 
2k-1\mod {(p-1)}$, the 
$L$-function formula for $\pi_{2n}Y_{2k}$ still holds, and
$\pi_{2n+1}Y_{2k}=0$ still holds provided the value of the
$L$-function is non-zero.  If the value of the $L$-function is zero,
then $\pi_{2n+1}Y_{2k}\iso \bZp$, though it is
conjectured~\cite{Schneider,Kolsteretal} that this case never occurs.)  

For $p$ not satisfying the Kummer-Vandiver condition, the odd
summands satisfy 
\begin{align*}
&\pi_{2n}Y_{2k+1}=\text{finite}\\
&\pi_{2n+1}Y_{2k+1}\iso \bZp
\end{align*}
(non-canonical isomorphism)
for\footnote{The published version has $2k+1$ here, which is incorrect} $n\equiv 2k\mod {(p-1)}$
(and zero otherwise) with the finite group unknown.
As always $Y_{0}\simeq *$, and the Mazur-Wiles
theorem~\cite{MazurWiles}, \cite[15.14]{Washington-1997} implies that
in the even summands, 
\begin{align*}
&\#(\pi_{2n}Y_{2k})=\#\left(\bZp/\biggl(\frac{\scriptstyle B_{n+1}}{\scriptstyle n+1}\biggr)\right)\\
&\pi_{2n+1}Y_{2k}=0
\end{align*}
for $n>0$, $n\equiv 2k-1\mod {(p-1)}$
(and $\pi_{2n}Y_{2k}=0$ for $n\not\equiv 2k-1\mod {(p-1)}$), although the precise
group in the first case is unknown.  (In this case, for $n<0$, $n\equiv
2k-1\mod {(p-1)}$, it is known that $\#(\pi_{2n}Y_{2k})=
\#(\bZp/L_{p}(-n,\omega^{2k}))$ and 
$\pi_{2n+1}Y_{2k}=0$, provided $L_{p}(-n,\omega^{2k})$ is non-zero. If
$L_{p}(-n,\omega^{2k})= 0$, then $\pi_{2n}Y_{2k}\iso
\bZp\oplus\text{finite}$ and $\pi_{2n+1}Y_{2k}\iso\bZp$,
non-canonically.) 
For more on the homotopy groups of $K(\bZ)$, see for example~\cite[\S
VI.10]{Weibel-KBook}. 

\subsection*{Supporting calculations}

In several places above, we claimed (implicitly or explicitly) that
certain hom sets in the stable category were zero.  Here we review
some calculations and justify these claims.  All of these computations
follow from well-known facts about the spectrum $L$ together with
standard facts about maps in the stable category. In particular, in
several places, we make use of the fact that for a $K(1)$-local
spectrum $Z$, the localization map $X \to L_{K(1)}X$ induces an
isomorphism $[L_{K(1)}X,Z] \to [X,Z]$; also, several times we make use
of the fact that if $X$ is $(n-1)$-connected, then the
$(n-1)$-connected cover map $Z[n,\infty)\to Z$ induces an isomorphism
$[X,Z[n,\infty)]\to [X,Z]$.  We begin with results on
$[\ell,\Sigma^{q}\ell]$.

\begin{prop}\label{prop:ellell}
The map $[\ell,\Sigma^{q}\ell]\to [\ell,\Sigma^{q}L]\iso
[L,\Sigma^{q}L]$ is an injection for $q\leq 2(2p-2)$.  In particular,
$[\ell,\Sigma^{q}\ell]=0$ if $q\not\equiv 0\mod{(2p-2)}$ and $q<2(2p-2)$.
\end{prop}

\begin{proof}
We have a cofiber sequence 
\[ 
\Sigma^{q-1-(2p-2)}H\bZp\to\Sigma^{q}\ell\to \Sigma^{q-(2p-2)}\ell\to
\Sigma^{q-(2p-2)}H\bZp, 
\] 
and a corresponding long exact sequence
\[
\cdots \to [\ell,\Sigma^{q-1-(2p-2)}H\bZp]
\to [\ell,\Sigma^{q}\ell]\to [\ell,\Sigma^{q-(2p-2)}\ell]\to \cdots.
\]
First we note that the map $[\ell,\Sigma^{q}\ell]\to
[\ell,\Sigma^{q-(2p-2)}\ell]$ is injective for $q\leq 2(2p-2)$:  When
$q\neq (2p-2)+1$ this follows from the fact that
\[
[\ell,\Sigma^{q-1-(2p-2)}H\bZp]=H^{q-1-(2p-2)}(\ell;\bZp)=0
\]
for $q-1-(2p-2)<(2p-2)$ unless $q-1-(2p-2)=0$.  In the case $q=(2p-2)+1$, the image of
$[\ell,\Sigma^{0}H\bZp]$ in $[\ell,\Sigma^{q}\ell]$ in the
long exact sequence is still zero because the map 
$[\ell,\ell]\to [\ell,H\bZp]\iso \bZp$ is surjective.
Now when $q-(2p-2)<2p-2$, 
\[
[\ell,\Sigma^{q-(2p-2)}\ell]\iso
[\ell,\Sigma^{q-(2p-2)}L]\iso [L,\Sigma^{q-(2p-2)}L]
\]
since then
$\Sigma^{q-(2p-2)}\ell\to \Sigma^{q-(2p-2)}L$ is a weak equivalence on 
connective covers.  For the remaining case $q=2(2p-2)$, we have seen
that the map $[\ell,\Sigma^{q}\ell]\to [\ell,\Sigma^{q-(2p-2)}\ell]$
is an injection and the map 
$[\ell,\Sigma^{q-(2p-2)}\ell]\to [\ell,\Sigma^{q-(2p-2)}L]$ is an
injection.
\end{proof}

Next, using the cofiber sequence 
\[
\Sigma^{-1}\ell\to \Sigma^{(2p-2)-1}\ell\to j\to \ell
\]
and applying the previous result, we obtain the following
calculation. 

\begin{prop}\label{prop:jell}
$[j,\Sigma^{q}\ell]=0$ if $q\not\equiv 0\mod{(2p-2)}$ and $q<2(2p-2)$.
\end{prop}

\begin{proof}
Looking at the long exact sequence 
\[
\cdots \to [\ell,\Sigma^{q}\ell]\to [j,\Sigma^{q}\ell] \to
[\Sigma^{(2p-2)-1}\ell,\Sigma^{q}\ell]\to [\Sigma^{-1}\ell,\Sigma^{q}\ell]\to\cdots 
\]
and using the isomorphism $[\Sigma^{(2p-2)-1}\ell,\Sigma^{q}\ell]\iso
[\ell,\Sigma^{q+1-(2p-2)}\ell]$ we have that both
$[\ell,\Sigma^{q}\ell]$ and $[\Sigma^{(2p-2)-1}\ell,\Sigma^{q}\ell]$
are 0 when $q\not\equiv 0,-1\mod{(2p-2)}$ and $q\leq 2(2p-2)$.  In the
case when $q\equiv -1\mod{(2p-2)}$, using also the isomorphism 
$[\Sigma^{-1}\ell,\Sigma^{q}\ell]\iso [\ell,\Sigma^{q+1}\ell]$, we
have a commutative diagram 
\[
\xymatrix{%
[\ell,\Sigma^{q+1-(2p-2)}\ell]\ar[r]\ar@{>->}[d]
&[\ell,\Sigma^{q+1}\ell]\ar@{>->}[d]\\
[L,\Sigma^{q+1-(2p-2)}L]\ar@{>->}[r]
&[L,\Sigma^{q+1}L]
}
\]
where the feathered arrows are known to be injections.  The statement
now follows in this case as well.
\end{proof}

For maps the other way, we have the following result.  The proof is
similar to the proof of the previous proposition.

\begin{prop}\label{prop:ellj}
$[\Sigma^{q}\ell,j]=0$ if $q\not\equiv -1\mod{(2p-2)}$ and $q\geq -(2p-2)$.
\end{prop}

We also have the following result for $q=-1$.

\begin{prop}\label{prop:elljmo}
$[\Sigma^{-1}\ell,j]=0$
\end{prop}

\begin{proof}
Let $j_{-1}=J[-1,\infty)$ where $J=L_{K(1)}\bS\simeq L_{K(1)}j$.  Then
we have a cofiber sequence $\Sigma^{-2}H\pi_{-1}J\to j\to j_{-1}\to \Sigma^{-1}H\pi_{-1}J$ and a long exact
sequence 
\[
\cdots \to [\Sigma^{-1}\ell,\Sigma^{-2}H\pi_{-1}J]\to [\Sigma^{-1}\ell,j]\to 
[\Sigma^{-1}\ell,j_{-1}]\to [\Sigma^{-1}\ell,\Sigma^{-1}H\pi_{-1}J]\to \cdots .
\]
Since $\Sigma^{-1}\ell$ is $(-2)$-connected,
the inclusion of $j_{-1}$ in
$J$ induces a bijection 
\[
[\Sigma^{-1}\ell,j_{-1}]\to [\Sigma^{-1}\ell,J]\iso [\Sigma^{-1}L,J].
\]
It follows that a map $\Sigma^{-1}\ell\to j_{-1}$ is determined by the
map on $\pi_{-1}$, and therefore that the image of 
$[\Sigma^{-1}\ell,j]$ in $[\Sigma^{-1}\ell,j_{-1}]$ is zero.  But 
$H^{-2}(\Sigma^{-1}\ell;\pi_{-1}J)=0$, so  $[\Sigma^{-1}\ell,j]=0$.
\end{proof}

In the case of maps between suspensions of $j$, we only need to
consider two cases:

\begin{prop}\label{prop:jj}
$[j,\Sigma j]=0$ and $[\Sigma j,j]=0$.
\end{prop}

\begin{proof}
As in the previous proof, we let $j_{-1}=J[-1,\infty)$, and we use the 
cofiber sequence $\Sigma^{-1}H\pi_{-1}J\to \Sigma j\to \Sigma j_{-1}\to H\pi_{-1}J$ and the
induced long exact sequence 
\[
\cdots \to [j,\Sigma^{-1}H\pi_{-1}J]\to [j,\Sigma j]\to 
[j,\Sigma j_{-1}]\to [j,H\pi_{-1}J]\to \cdots .
\]
Since the connective cover of $\Sigma J$ is $\Sigma j_{-1}$, we have
that the map $[j,\Sigma j_{-1}]\to [j,\Sigma J]$ is a bijection, and
the maps
\[
[J,\Sigma J]\to [j,\Sigma J]\to [\bS,\Sigma J]\iso\pi_{-1}J
\]
are isomorphisms.  It follows that the map $[j,\Sigma j_{-1}]\to
[j,H\pi_{-1}J]$ is an isomorphism.  Since
$[j,\Sigma^{-1}H\pi_{-1}J]=0$, this proves $[j,\Sigma j]=0$.  For the
other calculation, the map $j\to J$ induces a weak equivalence of
$1$-connected covers, and the induced map
\[
[\Sigma j,j]\to [\Sigma j,J]\iso [\Sigma J,J] \iso [\Sigma \bS,J]=\pi_{1}J=0
\]
is a bijection.
\end{proof}

The following propositions are now clear.

\begin{prop}\label{prop:TCsummands}
In the notation above, the summands $x(i)$ of $TC(\bZ)\phat$ satisfy
$[x(i),x(i')]=0$ for $i\neq i'$.
\end{prop}

\begin{prop}\label{prop:splitjp}
Let $k=0,1$.  In the notation above, $[x(i),\Sigma^{k} j]=0$ for
$i\neq k$ and the inclusion of $\Sigma^{k} j$ in $x(k)$ induces a
bijection $[x(k),\Sigma^{k} j]\to [\Sigma^{k} j,\Sigma^{k} j]$.
\end{prop}

Eliminating the summands where maps out of $j$ or $\Sigma j$ are
trivial, and looking at the connective and $0$-connected covers of
$K(1)$-localizations, we get the following proposition.

\begin{prop}\label{prop:jvspi}
The map $\bS\to j$ induces isomorphisms $[j,TC(\bZ)\phat]\iso
\pi_{0}TC(\bZ)\phat$ and $[\Sigma j,TC(\bZ)\phat]\iso
\pi_{1}TC(\bZ)\phat$. 
\end{prop}

For the summands of $K(\bZ)\phat$, we first consider the splitting of $j$.

\begin{prop}\label{prop:jtKZ}
$[j,\tK(\bZ)]=0$ and $[\tK(\bZ),j]=0$
\end{prop}

\begin{proof}
As indicated above $\tK(\bZ)\simeq y_{0}\vee \cdots \vee y_{p-2}$.  We
have that $y_{1}$ is (non-canonically) weakly equivalent to
$\Sigma^{2p-1}\ell$, and applying Propositions~\ref{prop:jell},
\ref{prop:ellj}, we see that $[j,y_{1}]=0$ and $[y_{1},j]=0$.  In
addition, $y_{2}\simeq *$ and $y_{0}\simeq *$.  For $2<i\leq p-2$,
$y_{i}$ is the fiber of a map from a finite wedge of copies of
$\Sigma^{2i-1}\ell$ to a finite wedge of copies of $\Sigma^{2i-1}\ell$.
Looking at the long exact sequences
\begin{gather*}
\cdots \to \bigoplus [j,\Sigma^{2i-2}\ell]\to [j,y_{i}]\to 
\bigoplus [j,\Sigma^{2i-1}\ell]\to \cdots\\
\cdots\to \prod [\Sigma^{2i-1}\ell,j]\to [y_{i},j]\to \prod
[\Sigma^{2i-2}\ell,j]\to \cdots,
\end{gather*}
we again see from Propositions~\ref{prop:jell} and~\ref{prop:ellj}
that $[j,y_{i}]=0$ and $[y_{i},j]=0$.
\end{proof}

Finally, the
spectra $Y_{i}$ clearly satisfy $[Y_{i},Y_{i'}]=0$ for $i\neq i'$; we now
verify that the same holds for the covers $y_{i}$.

\begin{prop}\label{prop:KZsummands}
In the notation above, the summands $y_{i}$ of $\tK(\bZ)$ satisfy
$[y_{i},y_{i'}]=0$ for $i\neq i'$.
\end{prop}

\begin{proof}
Each $y_{i}$ is the fiber of a map from a finite
wedge of copies of $\Sigma^{2i-1}\ell$ to a finite wedge of copies of
$\Sigma^{2i-1}\ell$ except that $y_{1}\simeq \Sigma^{2p-1}\ell$
(non-canonically), $y_{0}\simeq *$, and $y_{2}\simeq *$.  First, for
$i>2$, looking at the 
long exact sequence 
\[
\cdots\to \prod [\Sigma^{2i-1}\ell,\Sigma^{q}\ell]\to [y_{i},\Sigma^{q}\ell]\to \prod
[\Sigma^{2i-2}\ell,\Sigma^{q}\ell]\to \cdots,
\]
we see from Proposition~\ref{prop:ellell} that
$[y_{i},\Sigma^{q}\ell]=0$ when $q\not\equiv 2i-1,2i-2\mod{(2p-2)}$ and
$q\leq 2(2p-2)+2i-3$.  In particular, $[y_{i},y_{1}]=0$ for $i>2$.
For $i'>2$, looking at the long exact sequence 
\[
\cdots \to \bigoplus [y_{i},\Sigma^{2i'-2}\ell]\to [y_{i},y_{i'}]\to 
\bigoplus [y_{i},\Sigma^{2i'-1}\ell]\to \cdots,
\]
we see that $[y_{i},y_{i'}]=0$ for $i\neq i'$ in the remaining cases.
\end{proof}

\subsection*{The homotopy type of \texorpdfstring{$TC(\bZ)\phat$}{TC(Z)p}}\label{subsec:newsec}

To determine the homotopy type of $TC(\bZ)\phat$ from that of
$TC(\bZ)\phat[0,\infty )$, we need to study the cofiber sequence
\[
\Sigma^{-2}H\bZp\to TC(\bZ)\phat[0,\infty)\to
TC(\bZ)\phat\to \Sigma^{-1} H\bZp.
\]
The argument of \cite[3.3]{Rognesp} studies the induced cofiber
sequence on the cofiber of the inclusion of $j$, 
\[
\Sigma^{-2}H\bZp\to C\big(j\rightarrow TC(\bZ)\phat[0,\infty)\big)\to
C\big(j\rightarrow TC(\bZ)\phat\big)\to \Sigma^{-1} H\bZp,
\]
and shows that $C\big(j\rightarrow TC(\bZ)\phat\big)$ has the homotopy
type of
\[
\Sigma j\vee
\Sigma^{3}\ell \vee \cdots \vee \Sigma^{2p-5}\ell\vee \Sigma^{-1}\ell
\vee \Sigma^{2p-1}\ell.
\]
In terms of the non-canonical weak equivalence
\[
C\big(j\rightarrow TC(\bZ)\phat[0,\infty)\big)\simeq 
\Sigma j\vee \Sigma^{3}ku\phat,
\]
the map $\Sigma^{-2}H\bZp\to C\big(j\rightarrow
TC(\bZ)\phat[0,\infty)\big)$ factors through a map
\[
\Sigma^{-2}H\bZp\overto{\lambda} \Sigma^{2p-3}\ell\overto{\alpha} \Sigma^{3}ku\phat
\]
where $\alpha \colon \Sigma^{2p-3}\ell\to \Sigma^{3}ku\phat$ is the inclusion of an
Adams summand and the map $\lambda \colon \Sigma^{-2}H\bZp\to \Sigma^{2p-3}\ell$ is a
generator of $[\Sigma^{-2}H\bZp,\Sigma^{2p-3}\ell]\iso \bZp$.  It
follows in particular that pulling back along $\lambda$ induces an
isomorphism $[\Sigma^{2p-3}\ell,j]\to [\Sigma^{-2}H\bZp,j]$. 
The splitting 
\[
TC(\bZ)\phat[0,\infty)\simeq j\vee C\big(j\rightarrow TC(\bZ)\phat[0,\infty)\big)
\]
is non-canonical, determined by a choice of map
$TC(\bZ)\phat[0,\infty)\to j$ (in the stable category) such that the
composite map $j\to j$ is the identity; choosing an arbitrary such
map, we can alter it by a map $\Sigma^{2p-3}\ell\to j$ to get a
splitting where the composite map 
\[
\Sigma^{-2}H\bZp\to TC(\bZ)\phat[0,\infty)\to j
\]
is the zero map.  It follows that $TC(\bZ)\phat$ is non-canonically
weakly equivalent to
\[
j\vee \Sigma j\vee
\Sigma^{3}\ell \vee \cdots \vee \Sigma^{2p-5}\ell\vee \Sigma^{-1}\ell
\vee \Sigma^{2p-1}\ell.
\]

\section{The maps in the linearization/cyclotomic trace square}\label{sec:maps}

The previous section discussed the corners of the
linearization/cyclotomic trace square; in this section, we discuss the
edges.  The main observation is that with respect to the canonical
splittings of the previous section, the cyclotomic trace is diagonal
and the linearization map is diagonal on the $p$-torsion part of the homotopy groups.

\begin{thm}\label{thm:tracediag}
In terms of the splittings~\eqref{eq:TCZmain} and~\eqref{eq:KZmain} of
the previous section, the cyclotomic trace $K(\bZ)\phat\to
TC(\bZ)\phat$ splits as the wedge of the identity map $j\to j$ and maps
\begin{align*}
y_{0}&\to\Sigma^{-1}\ell_{TC}(0)\\
y_{1}&\to\Sigma^{-1}\ell_{TC}(p)\\
y_{2}&\to\Sigma^{-1}\ell_{TC}(2)\\
\vdots\ &\hphantom{\overto{=}}\ \ \vdots\\
y_{p-2}&\to\Sigma^{-1}\ell_{TC}(p-2).
\end{align*}
\end{thm}

\begin{proof}
We have that each $y_{i}$ fits into a fiber sequence of the form
\[
y_{i}\to
\bigvee \Sigma^{2i-1}\ell\to 
\bigvee \Sigma^{2i-1}\ell
\]
except in the case $i=1$ where the suspension is $\Sigma^{2p-1}\ell$
rather than $\Sigma^{1}\ell$
(and the cases $i=0$ and $i=2$, where $y_{i}=*$ anyway).  Choosing a
non-canonical weak equivalence $\Sigma^{2q-1}\ell\simeq
\Sigma^{-1}\ell_{TC}(q)$, and looking at the long exact sequences of
maps into $\Sigma^{2q-1}\ell$, Proposition~\ref{prop:ellell} implies
$[y_{i},\Sigma^{-1}\ell_{TC}(q)]=0$ unless $q \equiv i\mod {(p-1)}$.
Likewise, $[j,\Sigma^{-1}\ell_{TC}(q)]=0$ for all $q$ by
Proposition~\ref{prop:jell}. 
\end{proof}

Next, we turn to the linearization map.

\begin{thm}\label{thm:lindiag}
In terms of the splittings~\eqref{eq:TCS}, \eqref{eq:fubar},
and~\eqref{eq:TCZmain} of the previous section, the linearization 
map $TC(\bS)\phat\to TC(\bZ)\phat$ admits factorizations as follows:
\begin{enumerate}
\item The map $\bS\phat\to TC(\bZ)\phat$ factors through the canonical
map $\bS\phat\to j$, which is a canonically split surjection on
homotopy groups in all degrees.
\item The map $\Sigma \bS\phat\to TC(\bZ)\phat$ factors through a
map $\Sigma \bS\phat\to \Sigma j'$ that is an isomorphism on
$\pi_{1}$ and is a split surjection on homotopy groups in all
degrees. 
\item The map $\fubar\to TC(\bZ)\phat$ factors through
$j\vee \bigvee \Sigma^{-1}\ell_{TC}(i)$ (for $i=0$, $p$, $2$,\dots,
$p-2$).  On $p$-torsion, the map  
$\ptor(\pi_{*}(\fubar))\to\ptor(\pi_{*}(TC(\bZ)\phat))$ is
zero. 
\end{enumerate}
\end{thm}

\begin{proof}
The statement (i) is clear from the construction of the map $j\to
TC(\bZ)\phat$ since the linearization map is a map of ring spectra.
For (ii), the composite map $\Sigma \bS\phat\to TC(\bZ)\phat$ is
determined by where the generator goes in $\pi_{1}TC(\bZ)\phat$, but
the inclusion of $\Sigma j'$ in $TC(\bZ)\phat$ induces an isomorphism
on $\pi_{1}$, and so $\Sigma \bS\phat$ factors through $\Sigma j'$.
The map $TC(\bS)\phat\to TC(\bZ)\phat$ is a $(2p-3)$-equivalence~\cite[9.10]{BokstedtMadsen};
the map $\Sigma \bS\to \Sigma j'$ is therefore an isomorphism on
$\pi_{1}$.  Using the image of the generator of $\pi_{1}\Sigma\bS$
as a generator for $\pi_{1}\Sigma j'$ gives a weak equivalence $\Sigma
j\to \Sigma j'$ such that the map $\Sigma \bS\to \Sigma j'\to \Sigma
j$ obtained by composing with the inverse is the suspension of the
canonical map $\bS\to j$.    This completes the proof of (ii).

To show the factorization of $\fubar\to TC(\bZ)\phat$ for
(iii), it suffices 
to check that the composite map
\[
\fubar \to TC(\bZ)\phat\to \Sigma j'
\]
is trivial. Using the non-canonical weak equivalence $\fubar\simeq
\Sigma \bCPimbar$, the Atiyah-Hirzebruch spectral sequence to calculate
$(\Sigma j)^{*}(\Sigma \bCPimbar)$ with 
\[
E^{s,t}_{2}=H^{s+t}(\Sigma \bCPimbar;\pi_{-t}(\Sigma j'))
\iso H^{s+t}(\bCPimbar;\pi_{-t}(j))
\]
is zero along the line
of total degree zero, and so $[\fubar,\Sigma j']=0$.

Finally, to see that the map $\fubar\to TC(\bZ)\phat$ is
zero on the torsion subgroup of $\pi_{*}\fubar$, 
it suffices to note that the composite map 
\[
\Sigma \bCPimbar\simeq \fubar\to TC(\bZ)\phat\to j
\]
is zero on the torsion subgroup of
$\pi_{*}\Sigma \bCPimbar$, or equivalently that the composite map to $J$ is
zero on the torsion subgroup of $\pi_{*}\Sigma \bCPimbar$.
Post-composing with the (canonical) map $J\to L$ in the fiber sequence
\[
\Omega L\to J\to L\to L,
\]
the map $\Sigma
\bCPimbar\to L$ is trivial (since $L^{*}(\Sigma \bCPimbar)$ is
concentrated in odd degrees).  It follows that the map $\Sigma
\bCPimbar \to J$ factors through the map $\Omega L\to
J$, and therefore is zero on torsion.
\end{proof}

It would be reasonable to expect that the augmentations
$TC(\bS)\phat\to \bS\phat$ and $TC(\bZ)\phat \to j$ are compatible,
although we see no $K$-theoretic, $THH$-theoretic, or calculational
reasons why this should hold.  Such a compatibility would imply that
the map $\fubar\to TC(\bZ)\phat$ factors through $\bigvee
\Sigma^{-1} \ell_{TC}(i)$ and would then (combined with the observations in
Section~\ref{sec:splitting}) say that the linearization map is
fully diagonal with respect to the splittings of the previous
section. 

\section{Proof of main results}\label{sec:proof}

We now apply the work of the previous two sections to prove the
theorems stated in the introduction.  We begin with
Theorem~\ref{thm:mainles}, which is an immediate consequence of the
following theorem.

\begin{thm}\label{thm:surj}
The map
\[
\pi_{n}TC(\bS)\phat \oplus \pi_{n}K(\bZ)\phat
\to \pi_{n}TC(\bZ)\phat
\]
is (non-canonically) split surjective.
\end{thm}

We apply the splittings of $TC(\bS)\phat$, $TC(\bZ)\phat$, and
$K(\bZ)\phat$ and the maps on homotopy groups to prove the previous
theorem by breaking it into pieces and showing that different pieces
in the splitting induce surjections on homotopy groups.  
Theorem~\ref{thm:lindiag}.(i) and (ii) provide two surjection results;
we write two more in Lemmas~\ref{lem:klsurj} and~\ref{lem:y0}. The
first of these is essentially due to Klein-Rognes~\cite{KleinRognes}.

\begin{lem}\label{lem:klsurj}
Under the splittings~\eqref{eq:TCS}, \eqref{eq:fubar},
and~\eqref{eq:TCZmain}, the composite map
\[
\fubar\to TC(\bS)\phat\to TC(\bZ)\phat\to
\Sigma^{-1}\ell_{TC}(0)\vee \cdots \vee \Sigma^{-1}\ell_{TC}(p-2) 
\]
induces a split surjection on
$\pi_{2q+1}$ for $q\not\equiv 0\mod{(p-1)}$
\end{lem}

\begin{proof}
Klein and Rognes~\cite[5.8,(17)]{KleinRognes} (and independently
Madsen and Schlichtkrull~\cite[1.1]{MadsenSchlitkrull}) construct a
space-level map
\[
SU\phat\to \Omega^{\infty}(\Sigma \bCPim)\phat.
\]
They study the composite map 
\begin{equation}\label{eq:lprime}
SU\phat\to \Omega^{\infty}(\Sigma \bCPim)\phat \to
\Omega^{\infty}(TC(\bZ)\phat) \to
SU\phat
\end{equation}
induced by the linearization map $TC(\bS)\phat\to TC(\bZ)\phat$, the
projection map 
\[
TC(\bZ)\phat[0,\infty)\to \Sigma^{3}ku\phat, 
\]
and the Bott periodicity isomorphism $\Omega^{\infty}\Sigma^{3}ku\simeq
SU$.  In~\cite[6.3.(i)]{KleinRognes}, Klein and Rognes show that their
map~\eqref{eq:lprime}
induces an isomorphism of homotopy groups in all degrees except
those congruent to $1$ mod $2(p-1)$.  This proves the statement except
in degree $-1$, where it follows from the fact that the linearization
map $TC(\bS)\phat\to TC(\bZ)\phat$ is a $(2p-3)$-equivalence~\cite[9.10]{BokstedtMadsen}.
\end{proof}

The other lemma constructs a split surjection onto
$\pi_{*}TC(\bZ)\phat$ in degrees congruent to $1$ mod $2(p-1)$. 
We give a direct argument for the following lemma, but it can also be
proved using the vanishing results in~\cite{BM-Anil}.

\begin{lem}\label{lem:y0}
Under the splittings of~\eqref{eq:TCZmain} and~\eqref{eq:KZmain}, the
composite map
\[
y_{1}\to K(\bZ)\phat\to TC(\bZ)\phat\to \Sigma^{-1}\ell_{TC}(p)
\]
is a weak equivalence.
\end{lem}

\begin{proof}
Since $y_{1}$ and $\Sigma^{-1}\ell_{TC}(p)$ are both (non-canonically)
weakly equivalent to $\Sigma^{2p-1}\ell$, it suffices to show that the
map becomes a weak equivalence after $K(1)$-localization.  Indeed, 
by $v_{1}$ periodicity, it
suffices to show that the map on $K(1)$-localizations is an
isomorphism on any homotopy group in degree congruent to $1$ mod
$(2p-2)$.  We check this on $\pi_{1}$.  Consider the commutative
diagram
\[
\xymatrix@C-2pc@R-1pc{%
&L_{K(1)}K(\bZ[\tfrac1p])\ar[rr]
&&L_{K(1)}K(\bQ\phat)\\
L_{K(1)}K(\bZ)\ar[d]\ar[rr]\ar[ur]^-{\simeq}
&&L_{K(1)}K(\bZ\phat)\ar[d]^{\simeq}\ar[ur]_-{\simeq}\\
L_{K(1)}TC(\bZ)\ar[rr]_{\simeq}
&&L_{K(1)}TC(\bZ\phat)
}
\]
where the vertical maps are induced by the cyclotomic trace and the
indicated maps are weak equivalences, for the diagonal maps by Quillen's
localization sequence, for the indicated horizontal map by~\cite[Add.~6.2]{HM2}, and
for the indicated vertical map by~\cite[Th.~D]{HM2}.  On
$\pi_{1}$, the canonical maps
\begin{gather*}
\bZ[\tfrac1p]^{\times}\to \pi_{1}(K(\bZ[\tfrac1p]))\to \pi_{1}(L_{K(1)}K(\bZ[\tfrac1p]))\\
(\bQp)^{\times}\to \pi_{1}(K(\bQp))\to \pi_{1}(L_{K(1)}K(\bQp))
\end{gather*}
induce isomorphims
\begin{gather*}
(\bZ[\tfrac1p]^{\times})\phat\iso \pi_{1}(L_{K(1)}K(\bZ))\iso
\pi_{1}(Y_{1})\\
((\bQ\phat)^{\times})\phat\iso \pi_{1}(L_{K(1)}TC(\bZ))\iso
\pi_{1}(L_{K(1)}\Sigma j')\oplus \pi_{1}(L_{K(1)}\Sigma^{-1}\ell_{TC}(p)),
\end{gather*}
under which the map $\pi_{1}(Y_{1})\to \pi_{1}(L_{K(1)}TC(\bZ\phat))$
corresponds to the map induced by the inclusion of $\bZ[\tfrac1p]$ in
$\bQp$. 
By construction, the map $\Sigma j'\to TC(\bZ)\phat$ sends $\pi_{1}(j')$
isomorphically onto the subgroup $((\bZp)^{\times})\phat$ of
$((\bQ\phat)^{\times})\phat$ under the isomorphism above.  The map 
\[
(\bZ[\tfrac1p]^{\times})\phat\to ((\bQ\phat)^{\times})\phat/((\bZp)^{\times})\phat
\]
is an isomorphism of free $\bZp$-modules of rank 1.
\end{proof}

We now have everything we need for the proof of
Theorem~\ref{thm:surj}.

\begin{proof}[Proof of Theorem~\ref{thm:surj}]
Combining previous results, we have two families of (non-canonical)
splittings
\begin{multline*}
\pi_{*}(\bS\vee\Sigma \bS\vee \fubar\vee j\vee
y_{0}\vee\cdots \vee y_{p-2})\\
\to \pi_{*}(j\vee \Sigma j'\vee \Sigma^{-1}\ell_{TC}(0)\vee\cdots \vee \Sigma^{-1}\ell_{TC}(p-2).
\end{multline*}
For both splittings we use Lemma~\ref{lem:y0} to split the
$\Sigma^{-1}\ell_{TC}(p)$ summand in the codomain (canonically) using the $y_{1}$
summand of the domain, we use Theorem~\ref{thm:lindiag}.(ii) to split
the $\pi_{*}\Sigma j'$ summand in the codomain (canonically) using the
$\Sigma \bS^{1}$ summand in the domain, and we use
Lemma~\ref{lem:klsurj} to split the
\[
\pi_{*}(\Sigma^{-1}\ell_{TC}(0)\vee \Sigma^{-1}\ell_{TC}(2)\vee\cdots \vee \Sigma^{-1}\ell_{TC}(p-2))
\]
summands in the codomain (non-canonically) using the $\pi_{*}\fubar$ summand in the domain.  We then have a choice on the
remaining summand of the codomain, $\pi_{*}j$.  We can use
Theorem~\ref{thm:tracediag} to split this (canonically) using the
$\pi_{*}j$ summand in the domain or use Theorem~\ref{thm:lindiag} to
split this (canonically) using the $\pi_{*}\bS$ summand in the domain.
\end{proof}

This completes the proof of Theorem~\ref{thm:mainles}.  We now prove
the remaining theorems from the introduction.

\begin{proof}[Proof of Theorems~\ref{thm:main} and~\ref{thm:subgp}]
Since the long exact sequence of the homotopy cartesian
linearization/cyclotomic trace square breaks into split short exact
sequences, we get split short exact sequences on $p$-torsion subgroups
\[
0\sto \ptor(\pi_{n}K(\bS))\to 
\ptor(\pi_{n}TC(\bS)\phat\oplus \pi_{n}K(\bZ))\to
\ptor(\pi_{n}TC(\bZ)\phat)\to 0.
\]
Using the splittings of~\eqref{eq:TCS}, \eqref{eq:TCZmain},
and~\eqref{eq:KZ}, leaving out the non-torsion summands, we can
identify $\ptor(\pi_{n}K(\bS))$ as the kernel of a map 
\[
\ptor(\pi_{n}\bS\oplus \pi_{n-1}\bS\oplus \pi_{n}\fubar \oplus
\pi_{n}j \oplus \pi_{n}\tK(\bZ))
\to \ptor(\pi_{n}j\oplus \pi_{n-1}j')
\]
which by Theorems~\ref{thm:tracediag} and~\ref{thm:lindiag} is mostly
diagonal: It is the direct sum of the canonical maps 
\begin{align*}
\ptor(\pi_{n}\bS)\oplus \ptor(\pi_{n}j)&\to \ptor(\pi_{n}j)\\
\ptor(\pi_{n-1}\bS)&\to \ptor(\pi_{n-1}j')
\end{align*}
and the zero maps on $\ptor(\pi_{n}\fubar)$ and
$\ptor(\pi_{n}\tK(\bZ))$. 
The isomorphism~\eqref{a} uses the
canonical splitting $\pi_{n}j\to \pi_{n}\bS$ on the $\pi_{n}\bS$
summand, while the isomorphism~\eqref{b} uses the identity of
$\pi_{n}j$ on the $\pi_{n}j$ summand.
\end{proof}

\section{Conjecture on Adams Operations}\label{sec:splitting}

In Section~\ref{sec:review} we produced canonical splittings on
$K(\bZ)\phat$ and $TC(\bZ)\phat$ and in Section~\ref{sec:maps}, we
showed that the cyclotomic trace is diagonal with respect to these
splittings.  The purpose of this section is to prove the following
splitting of the linearization/cyclotomic trace square and relate it
to conjectures on Adams operations.

\begin{thm}\label{thm:eigensquare}
The spectrum $K(\bS)\phat$ splits into $p-1$ summands,
$K(\bS)\phat\simeq \aK_{0}\vee\cdots \vee\aK_{p-2}$, and the
linearization/cyclotomic trace square splits into the wedge sum of
$p-1$ homotopy cartesian squares 
\[
\xymatrix{%
\aK_{0}\ar[r]\ar[d]&j\ar[d]
&\aK_{1}\ar[r]\ar[d]&y_{1}\ar[d]\\
\bS\phat\vee \Sigma \bCPim[-1]\ar[r]&j\vee \Sigma^{-1}\ell_{TC}(0)
&\Sigma \bCPim[0]\ar[r]&\Sigma j\vee \Sigma^{-1}\ell_{TC}(p)
}
\]
for $i=0,1$, and
\[
\xymatrix{%
\aK_{i}\ar[r]\ar[d]&y_{i}\ar[d]\\
\Sigma \bCPim[i-1]\ar[r]&\Sigma^{-1}\ell_{TC}(i)
}
\]
for $i=2,\ldots p-2$.
\end{thm}

We are using the notation from Section~\ref{sec:review} for the
summands $y_{i}$ and $\Sigma^{-1}\ell_{TC}(i)$ of $K(\bZ)\phat$ and $TC(\bZ)\phat$.
The spectra $\bCPim[i]$ are the wedge summands of the ``Adams splitting'' 
previously used by Rognes~\cite[\S5]{Rognesp}
\begin{equation}\label{eq:cpimsplit}
\bCPim\simeq \bCPim[-1] \vee \bCPim[0]\vee \cdots \vee \bCPim[p-3].
\end{equation}
Here we are following the numbering of Rognes~\cite[p.~169]{Rognesp}, which
has its rationale in that the $[i]$ piece has its ordinary cohomology
concentrated in degrees $2i$ mod~${2p-2}$ and starts in degree $2i$.
Note that the splitting of the theorem fails to be canonical because
of problems with the identification of $TC(\bS)\phat$ as
$\bS\phat\vee \Sigma \bCPim$.

In the theorem, for the $i=0$ square, we have used that
$y_{0}=*$ as noted in Section~\ref{sec:review}.  In using these squares
to study the homotopy type of $\aK_{i}$, we can simplify the $i=0$
square to a cofiber sequence
\[
\aK_{0}\to
\bS\phat\vee \Sigma \bCPim[-1]\to \Sigma^{-1}\ell_{TC}(0) \to \Sigma \aK_{0}
\]
since Theorem~\ref{thm:tracediag} indicates that the map $j\to j\vee
\Sigma^{-1}\ell_{TC}(0)$ factors through the identity map $j\to j$.
For the $i=1$ square, the splitting of $\bCPim$ fits into a fiber
sequence with the splitting of 
\[
(\Sigma^{\infty}\bCPi)\phat\simeq \Sigma^{\infty}K(\bZp,2))\simeq
\bCPi[1]\vee \cdots \vee \bCPi[p-1],
\]
which identifies $\bCPim[0]$ as $\bS\phat\vee \bCPi[p-1]$.
Theorem~\ref{thm:tracediag} and Lemma~\ref{lem:y0}
indicate that the map $y_{1}\to \Sigma j\vee \Sigma^{-1}\ell_{TC}(p)$
factors through a weak equivalence $y_{1}\to \Sigma^{-1}\ell_{TC}(p)$, and we get
a weak equivalence 
\[
\aK_{1}\simeq \Sigma c\vee \bCPi[p-1],
\]
where by definition $c$ is the homotopy fiber of the canonical map
$\bS\to j$.

\begin{proof}[Proof of Theorem~\ref{thm:eigensquare}]
Let $\epsilon_{i}TC(\bS)\phat$, $\epsilon_{i}TC(\bZ)\phat$, and
$\epsilon_{i}K(\bZ)\phat$ be the summands of $TC(\bS)\phat$,
$TC(\bZ)\phat$, and $K(\bZ)\phat$, respectively, specified in the 
$i$th square in the statement of the theorem.  Our work in
Section~\ref{sec:maps} shows that the map $K(\bZ)\phat\to
TC(\bZ)\phat$ restricts to a sum of maps $\epsilon_{i}K(\bZ)\phat\to
\epsilon_{i}TC(\bZ)\phat$.  The Atiyah-Hirzebruch spectral sequence
implies that the $\bCPim[i-1]$ summand of
$\epsilon_{i}TC(\bS)\phat$ factors uniquely through
$\epsilon_{i}TC(\bZ)\phat$ as there are no essential maps to the other
summands; Theorem~\ref{thm:lindiag} then implies that the map
$TC(\bS)\phat\to TC(\bZ)\phat$ decomposes as a wedge sum of maps
$\epsilon_{i}TC(\bS)\phat\to \epsilon_{i}TC(\bZ)\phat$. 
\end{proof}

In the proof above, we argued calculationally using the paucity of
maps in the stable category between the summands; however, there is a
conceptual reason to expect much of this behavior based on
$p$-adically interpolated of Adams operations.
In~\cite[10.7]{GoodTHH}, we showed that the splitting of $\bCPimbar$
arises from a $p$-adic interpolation of the Adams operations on
$TC(\bS)\phat$ (constructed there).  Such a $p$-adic interpolation is
an extension to $\bZpt$ of the action of the monoid
$\bZ_{(p)}^{\times}\cap \bZ$ on $TC(-;p)$ (in the $p$-complete stable
category) acting by Adams operations.  Using the Teichm\"uller
character $\omega$ to embed $(\bZ/p)^{\times}$ in $\bZpt$, we then get
an action of the ring $\bZp[(\bZ/p)^{\times}]$, which then produces an
eigensplitting of $TC(\bS)\phat$ into the $p-1$ summands corresponding
to the powers of the Teichm\"uller character.  The wedge summand
$\epsilon_{i}TC(\bS)\phat$ in the proof of
Theorem~\ref{thm:eigensquare} corresponds to the character
$\omega^{i}$ in the sense that the $p$-adically interpolated Adams
operation $\psi^{\omega(\alpha)}$ acts on it by multiplication by
$\omega^{i}(\alpha)\in \bZp$ for all $\alpha \in (\bZ/p)^{\times}$.

We can imagine that $p$-adically interpolated Adams operations act on
the whole linearization/cyclotomic trace square; we would then obtain
an eigensplitting into $p-1$ squares exactly as in the statement of
Theorem~\ref{thm:eigensquare}.  We regard this as providing evidence
for the following pair of conjectures, which together would give a
conceptual (as opposed to calculational) proof of
Theorem~\ref{thm:eigensquare}. 

\begin{conj}\label{conj:Adams}
Let $R$ be an $E_{\infty}$ ring spectrum.  There exists a 
homomorphism from $\bZpt$ to the composition monoid
$[K(R)\phat,K(R)\phat]$, which is natural in the obvious sense and
satisfies the following properties.
\begin{enumerate}
\item When $R$ is a ring, the restriction to $\bZ\cap \bZpt$ gives
Quillen's Adams operations on the zeroth space.
\item The induced map $\bZpt\to
\Hom(\pi_{*}K(R)\phat,\pi_{*}K(R)\phat)$ is continuous where the
target is given the $p$-adic topology.
\end{enumerate}
\end{conj}

\begin{conj}\label{conj:AdamsCyc}
For $R$ a connective $E_{\infty}$ ring spectrum, the cyclotomic trace\break
$K(R)\phat\to TC(R)\phat$ commutes with the (conjectural) $p$-adically
interpolated Adams operations.
\end{conj}

The preceding conjecture on $p$-adic interpolation of the Adams
operations in $p$-completed algebraic $K$-theory is weaker than the
(known) results in the case of topological $K$-theory in that we are
only asking for continuity on homotopy groups rather than continuity
for (some topology on) endomorphisms.  Nevertheless, it implies a
natural action of $\bZp[(\bZ/p)^{\times}]$ (in the stable category) on
$p$-completed algebraic $K$-theory spectra (of connective ring
spectra), which is all that is needed for the eigensplitting.  To
compare to the splitting used in Theorem~\ref{thm:eigensquare}, note
that for an algebraically closed field of characteristic prime to $p$
or a strict Henselian ring $A$ with $p$ invertible, the Adams
operation $\psi^{k}$ acts on $\pi_{2s}L_{K(1)}K(A)$ by multiplication
by $k^{s}$.  As a consequence, for any scheme satisfying the
hypotheses for Thomason's spectral sequence~\cite[4.1]{ThomasonEtale},
the eigensplitting on homotopy groups is compatible with the
filtration from $E_{\infty}$ in the sense that the subquotient of
$H^{s}(R;\bZ/p^{n}(i))$ comes from the $\omega^{i}$ summand.  For
$R=\bZ[1/p]$, we see that the $\omega^{i}$ summand of $L_{K(1)}K(R)$
has homotopy groups only in degrees congruent to $2i-1$ and $2i-2$ mod
$2(p-1)$ except when $i=0$ where the unit of $\pi_{0}L_{K(1)}(R)$ is
also in the trivial character summand.  As a consequence, we see that
this splitting agrees with the splitting described above for
$L_{K(1)}K(\bZ[1/p])\simeq L_{K(1)}K(\bZ)$.  Specifically, the summand
corresponding to trivial character is $J$ and the summand
corresponding to $\omega^{i}$ is $Y_{i}$ for $i=1,\ldots,p-2$.

The preceding conjectures also leads to the same splitting of
$TC(\bZ)\phat$.  Hesselholt-Madsen~\cite[Th.~D,Add.~6.2]{HM2} shows
that the completion map and cyclotomic trace
\[
TC(\bZ)\phat\to TC(\bZp)\phat\from K(\bZp)\phat
\]
are weak equivalences after taking the connective cover and so in
particular induce weak equivalences after $K(1)$-localization.  The
Quillen localization sequence identifies the homotopy fiber of the map
$K(\bZp)\to K(\bQp)$ as $K(\bF_{p})$.  Since the $p$-completion of
$K(\bF_{p})$ is weakly equivalent to $H\bZp$, its $K(1)$-localization
is trivial.  Combining these maps, we obtain a canonical isomorphism
in the stable category from $L_{K(1)}TC(\bZ)$ to $L_{K(1)}K(\bQp)$.
The Hesselholt-Madsen proof of the Quillen-Lichtenbaum conjecture for
certain local fields~\cite[Th.~A]{HMAnnals} (or in this case,
inspection from the calculation of $TC(\bZ)\phat$), shows that the map
\[
TC(\bZ)\phat\to L_{K(1)}TC(\bZ)\simeq L_{K(1)}K(\bQp)
\]
becomes a weak equivalence after taking $1$-connected covers.  Again
looking at Thomason's spectral sequence, we see that the
conjectural Adams operations would then split $L_{K(1)}K(\bQp)$ into 
summands as follows.  The summand
corresponding to the trivial character is $J\vee
\Sigma^{-1}L_{TC}(0)$, the summand corresponding to $\omega$ is $\Sigma
J \vee \Sigma^{-1}L_{TC}(1)$, and the summand corresponding to
$\omega^{i}$ is $\Sigma^{-1}L_{TC}(i)$ for $i=2,\ldots,p-2$.  A short
argument now shows that the eigensplitting gives the splitting of
$TC(\bZ)\phat$ used in Theorem~\ref{thm:eigensquare}. 

We have proposed relatively strong conjectures in~\ref{conj:Adams}
and~\ref{conj:AdamsCyc}; for a proof of Theorem~\ref{thm:eigensquare},
it would be enough for the operations to exist and be compatible just
for regular rings.  The work of Riou~\cite{Riou-AdamsOperations} makes
the existence of such operations plausible.  But given the work of
Dundas~\cite{Dundas-RelK}, the more general conjectures above (at
least for connective ring spectra $R$ with $\pi_{0}R$ regular) are not
far removed from the corresponding conjectures for rings.  We do also
note that work on non-existence of
determinants~\cite[2.3]{AusoniDundasRognes-Monopole} and \cite[Proof
of~3.7]{Waldhausen-ManifoldApproach} is often cited as evidence
against the existence of Adams operations on the algebraic $K$-theory
of ring spectra.

\section{Low degree computations}\label{sec:mess}

\bgroup
\newcommand{\tablesize}{\textstyle}
\newcommand{\tablessize}{\scriptstyle}
\newcommand{\ct}[1]{{\tiny\raisebox{1ex}{\textcircled{\raisebox{-.3ex}{#1}}}}}
\newcommand{\BMZ}{$\tablesize\bZ$}
\newcommand{\mz}[1]{\tablesize{\bZ}/{#1}}
\newcommand{\mzp}[2]{\tablesize(\mz{#1})^{\tablessize #2}}
\newcommand{\p}{&\hbox to 0pt{\hss${\tablesize\oplus}$\hss}&}
\newcommand{\T}{\mathbin{\tablesize\times}}
\newcommand{\ug}[1][{?}]{\tablesize{\relax[#1\relax]}}
\newcommand{\kt}[1]{$\tablesize K_{#1}(\bZ)$}
\newcommand{\lf}[1]{\ifcase#1
$\mz{16}\T\mz3\T\mz5$
\or$\mz8\T\mz9\T\mz7$
\or$\mz{32}\T\mz2\T\mz3\T\mz5$
\or$\mz8\T\mz2\T\mz3\T\mz{11}$
\or$\mz8\T\mzp22$
\or$\mz{32}\T\mzp23$
\fi}

\begin{table}[p]
\caption{The homotopy groups of $K(\bS)$ in low degrees}
\label{table}
\scalebox{0.833}{%
\begin{tabular}{r|l*{6}{cl}}
\hline
$n$&\multicolumn{10}{c}{$\pi_{n}K(\bS)$}\vrule height1.25em width0pt depth.5em& \\
\hline\hline\vrule height3ex width0pt depth0pt
 0&\BMZ&&            &&            &&      &&            &&\\[.5ex]
 1&  &&$\mz2$      &&            &&      &&            &&\\[.5ex]
 2&  &&$\mz2$      &&            &&      &&            &&\\[.5ex]
 3&  &&$\mz8\T\mz3$\p$\mz2$      &&      &&            &&\\[.5ex]
 4&$\tablesize0$\\[.5ex]
 5&\BMZ&&            &&            &&      &&            &&\\[.5ex]
 6&  &&$\mz2$      &&            &&      &&            &&\\[.5ex]
 7&  &&\lf0        \p$\mz2$      &&      &&            &&\\[.5ex]
 8&  &&$\mzp22$    &&            &&      &&            \p\kt8\\[.5ex]
 9&\BMZ\p$\mzp23$    \p$\mz2$      &&      &&            &&\\[.5ex]
10&  &&$\mz2\T\mz3$\p\lf4        &&      &&            &&\\[.5ex]
11&  &&\lf1        \p$\mz2$      \p$\mz3$&&            &&\\[.5ex]
12&  &&            \p$\mz4$      &&      &&            \p\kt{12}\\[.5ex]
13&\BMZ\p$\mz3$      &&            &&      &&            &&\\[.5ex]
14&  &&$\mzp22$    \p$\mz4$      \p$\mz3$\p$\mz9$    &&\\[.5ex]
15&  &&\lf2        \p$\mzp22$    &&      &&            &&\\[.5ex]
16&  &&$\mzp22$    \p$\mz8\T\mz2$&&      \p$\mz3$      \p\kt{16}\\[.5ex]
17&\BMZ\p$\mzp24$    \p$\mzp22$    &&      &&            &&\\[.5ex]
18&  &&$\mz8\T\mz2$\p\lf5        &&      \p$\mz3\T\mz5$&&\\[.5ex]
19&  &&\lf3        \p$\ug[64]$   &&      &&            &&\\[.5ex]
20&  &&$\mz8\T\mz3$\p$\ug[128]$  &&      \p$\mz3$      \p\kt{20}\\[.5ex]
21&\BMZ\p$\mzp22$    \p$\ug[16]$   \p$\mz3$&&            &&\\[.5ex]
22&  &&$\mzp22$    \p$\ug[2^{?}]$&&      \p$\mz3$      \p$\mz{691}$\\[.5ex]
\hline
\multicolumn{1}{l}{}&\multicolumn{1}{l}{\ct{1}} && \multicolumn{1}{l}{\ct{2}} && \multicolumn{1}{l}{\ct{3}} && \multicolumn{1}{l}{\ct{4}} && \multicolumn{1}{l}{\ct{5}} && \multicolumn{1}{l}{\ct{6}}
     \\
\end{tabular}}

\smallskip

\begin{minipage}{\hsize}
\small
The table compiles the results reviewed in
Proposition~\ref{prop:calcsum} (q.v.~for sources) and
\cite[5.8]{Rognes2} into the computation of $\pi_{n}K(\bS)$ for $n\leq
22$.  The description of $\pi_{n}K(\bS)$ is divided into columns:
\renewcommand{\labelenumi}{\arabic{enumi}.}
\renewcommand{\labelenumii}{\arabic{enumii}.}
\begin{enumerate}
\item The non-torsion part
\item The contribution
from the torsion of $\bS$ 
\item The remaining $2$-torsion
(from \cite{Rognes2})
\item The contribution from
$\Sigma c$ for odd primes
\item The torsion contribution
from $\fubar$ for odd primes
\item The torsion
contribution from $\tK(\bZ)$ for odd primes
\end{enumerate}

\smallskip

Presently $K_{4n}(\bZ)$ is unknown for $n>1$, conjectured to be 0
(the Kummer-Vandiver conjecture) and if non-zero is a finite group with
order a product of irregular primes, each of which is $>10^{8}$.

\smallskip

Summands denoted as $[m]$ are finite groups of order $m$ whose
isomorphism class is not known.
\end{minipage}
\end{table}
\egroup

Theorem~\ref{thm:main} describes the $p$-torsion in $K(\bS)$ in terms
of the $p$-torsion in various pieces.  For convenience, we review in
Proposition~\ref{prop:calcsum} below what is known about the homotopy
groups of these pieces at least up to the range in which \cite{Rognesp} describes
the homotopy groups of $\bCPimbar\simeq \Sigma^{-1}\fubar$.  As a consequence of
Theorem~\ref{thm:main}, irregular primes potentially contribute in
degrees divisible by 4 but otherwise make no contribution to the
torsion of $K(\bS)$ until degree $22$.  Thus, $\pi_{*}K(\bS)$ in
degrees $\leq 21$ not divisible by 4 is fully computed (up to some
2-torsion extensions) by the work of
Rognes~\cite{Rognes2,Rognesp}.  For convenience, 
we assemble the computation of 
$\pi_{*} K(\bS)$ for $*\leq 22$ in Table~\ref{table}
on page~\pageref{table}.

\begin{prop}\label{prop:calcsum}
The $p$-torsion groups $\ptor(\pi_{*}\bS)$, $\ptor(\pi_{*}c)$, 
$\ptor(\pi_{*}\tK(\bZ))$, and
$\ptor(\pi_{*}\fubar)$ are known in at least the following ranges,
as follows.
\begin{enumerate}
\item $\pi_{*}\bS$, $\pi_{*}j$, see for example
\cite[1.1.13]{Ravenel-Green}. 
$\pi_{*}\bS$ splits as
\[
\pi_{*}\bS=\pi_{*}j\oplus \pi_{*}c.
\]
$\ptor(\pi_{k}j)$
is zero unless $2(p-1)$ divides $k+1$, in which case it is cyclic of order
$p^{s+1}$ where 
$k+1=2(p-1)p^{s}m$
for $m$ relatively prime to $p$.
See below for $\pi_{*}c$.
\item $\pi_{*}c$, see for example \cite[1.1.14]{Ravenel-Green}.
In degrees $\leq 6p(p-1)-6$, $\ptor(\pi_{*}c)$ is $\bZ/p$ in the
following degrees and zero in all others. (In the table, $\alpha_{1}\in
\pi_{2p-3}j$.)\par
\bigskip
\begin{tabular}{ll}
\hline
Generator&Degree\vrule height1.25em width0pt depth.5em \\
\hline\hline\vrule height3ex width0pt depth0pt
$\beta_{1}$                &$2p(p-1)-2$     \\
$\alpha_{1}\beta_{1}$      &$2(p+1)(p-1)-3$ \\
$\beta_{1}^{2}$            &$4p(p-1)-4$     \\
$\alpha_{1}\beta_{1}^{2}$  &$2(2p+1)(p-1)-5$\\
$\beta_{2}$                &$2(2p+1)(p-1)-2$\\
$\alpha_{1}\beta_{2}$      &$4(p+1)(p-1)-3$ \\
$\beta_{1}^{3}$            &$6p(p-1)-6$     \\
\end{tabular}	                              
\bigskip

\item $\pi_{*}\tK(\bZ)$, see for example \cite[\S VI.10]{Weibel-KBook} or 
Section~\ref{sec:review}.
$\ptor(\pi_{*}\tK(\bZ))$ is zero in odd degrees. 
If $p$ is regular, then $\ptor(\pi_{2k}\tK(\bZ))=0$.  If $p$ satisfies
the Kummer-Vandiver condition, then 
\[
\ptor(\pi_{4k}\tK(\bZ))=0 \quad \textrm{and} \quad \ptor(\pi_{4k+2}\tK(\bZ))=\bZp/(\tfrac{B_{2k+2}}{2k+2}),
\]
where $B_{n}$ denotes the Bernoulli number, numbered by the convention
$\tfrac{t}{e^{t}-1}=\sum B_{n}\tfrac{t^{n}}{n!}$.   If $p$ does not
satisfy the Kummer-Vandiver condition then 
$\ptor(\pi_{4k}\tK(\bZ))=0$ for $k=1$ and is an unknown finite group
for $k>1$, while $\ptor(\pi_{4k+2}\tK(\bZ))$ is an unknown group of
order $\#(\bZp/(B_{2k+2}/(2k+2))$ for all $k$.  
\smallskip

\item $\pi_{*}\fubar$, see \cite[4.7]{Rognesp}.
\begin{itemize}
\item In odd degrees $\leq |\beta_{2}|-2 = 2(2p+1)(p-1)-4$,
\[
\ptor(\pi_{2n+1}\fubar)=\bZ/p
\]
in degrees $n=p^{2}-p-1+m$ or $n=2p^{2}-2p-2+m$ for $1\leq m\leq p-3$
and zero otherwise.  

\item In even degrees $\leq 2p(p-1)$, 
\[
\ptor(\pi_{2n}\fubar)=\bZ/p
\]
for $m(p-1)< n< mp$ for $2\leq m\leq p-1$ and zero otherwise except that 
\[
\ptor(\pi_{2(p(p-1)-1)}\fubar)=0.
\] 

Beyond this range, \cite[4.7]{Rognesp} only describes the size of the
group.  In all even degrees $\leq |\beta_{2}|-2 = 2(2p+1)(p-1)-4$, 
the size of $\ptor(\pi_{2n}\fubar)$ is 
\[
\#(\ptor(\pi_{2n}\fubar))=p^{a(n)+b(n)-c(n)-d(n)+e(n)}
\]
where
\begin{align*}
a(n)&=\lfloor (n-1)/(p-1) \rfloor\\
b(n)&=\lfloor (n-1)/(p(p-1)) \rfloor\\
c(n)&=\lfloor n/p \rfloor\\
d(n)&=\lfloor n/p^{2} \rfloor
\end{align*}
and 
\[
e(n) = 
\begin{cases}
  +1, & \text{if } n=p^{2}-2+mp \text{ for } 1\leq m\leq p-3, \\
  -1, & \text{if } n=p-1+mp \text{ for } m\geq p-2 \\
   0, & \text{otherwise}.
\end{cases}
\]
Here $\lfloor x\rfloor$ denotes the greatest integer $\leq x$. The
formulas $a,b,c,d$ above count the number of positive integers $< n$
(in $a$ and $b$) or $\leq n$ (in $c$ and $d$) that are divisible by
the denominator. 

\medskip

The formula above shows that for $p=3$, $\pi_{14}(\fubar)$
has order 9.  Rognes~\cite[4.9(b)]{Rognesp} shows that this group is
$\bZ/9$. 

\end{itemize}
\end{enumerate}
\end{prop}


\FloatBarrier

\bibliographystyle{plain}

\end{document}